\documentclass[12pt]{amsart}
\usepackage{a4wide, amsmath, amssymb}
\usepackage[usenames]{color}
\newtheorem{thm}{Theorem}[section]
\newtheorem{cor}[thm]{Corollary}
\newtheorem{lem}[thm]{Lemma}
\newtheorem{prop}[thm]{Proposition}
\theoremstyle{definition}
\theoremstyle{definition}
\theoremstyle{definition}
\theoremstyle{definition}

\newtheorem{rem}[thm]{Remark} \numberwithin{equation}{section}

\newcommand{\R}{\mathbb R}
\newcommand{\Z}{\mathbb Z}

\newcommand{\T}{\mathbb T}
\newcommand{\N}{\mathbb N}

\def\1{\mathbb I}

\def\e{\epsilon}

\begin{document}
\title[Lipschitz regularity results for nonlinear strictly elliptic equations]
{Lipschitz regularity results for nonlinear strictly elliptic equations and applications}

\author{Olivier Ley\and Vinh Duc Nguyen}
\address{IRMAR, INSA de Rennes, 35708 Rennes, France} \email{olivier.ley@insa-rennes.fr}
\address{School of Mathematics, University of Cardiff, UK} \email{nguyenv2@cardiff.ac.uk
}

\begin{abstract}
Most of lipschitz regularity results for nonlinear strictly elliptic equations are obtained for a 
suitable growth power of the nonlinearity
with respect to the gradient variable (subquadratic for instance). For equations with 
superquadratic growth power in gradient, one usually uses weak Bernstein-type
arguments which require regularity and/or convex-type assumptions on
the gradient nonlinearity.  In this article, we obtain new Lipschitz regularity results for  
a large class of nonlinear strictly elliptic equations with possibly arbitrary growth power 
of the Hamiltonian with respect to the gradient variable
using some ideas coming from Ishii-Lions' method. 
We use these bounds to solve an ergodic problem and to
study the regularity and the large time behavior of the solution
of the evolution equation. 
\end{abstract}


\date{\today}

\maketitle


\section{Introduction}

The main goal of this work is to obtain gradient bounds, which are uniform in $\e>0$ and $t$
respectively,
for the viscosity solutions of
a large class of nonlinear strictly elliptic equations
\begin{eqnarray}\label{approx-scal}
\e v^\e-{\rm trace}(A(x) D^2v^\e)
+  H(x, Dv^\e)=0, & x\in\T^N,
\end{eqnarray}
and 
\begin{eqnarray}\label{edp-evol}
\left\{
\begin{array}{ll}
\displaystyle \frac{\partial u}{\partial t}-{\rm trace}(A(x) D^2u)
+  H(x, Du)=0, & (x,t)\in\T^N\times (0,+\infty),\\[2mm]
u(x,0)=u_0(x), &  x\in\T^N.
\end{array}
\right.
\end{eqnarray}
We work in the periodic setting ($\T^N$ denotes the flat torus $\R^N/\Z^N$)
and assume for simplicity that $A(x)=\sigma(x)\sigma(x)^T$ with $\sigma\in W^{1,\infty}(\T^N;\mathcal{M}_N).$ 
Let us mention that  all the results of this paper hold true if $\sigma \in C^{0,1/2}(\T^N;\mathcal{M}_N)$.

We recall that a diffusion matrix $A$ is called strictly elliptic if
\begin{eqnarray}\label{sig-deg}
&& \begin{array}{l}
\text{there exists } \nu >0 \text{ such that }
A(x)\geq \nu I, \quad
x\in\T^N.
\end{array}
\end{eqnarray}

Most of Lipschitz regularity results for elliptic equations are obtained 
for a suitable growth power with respect to the gradient variable (subquadratic for instance, see
Frehse~\cite{frehse81}, Gilbarg-Trudinger~\cite{gt98}). 
In this article, we establish some gradient bounds
\begin{eqnarray}
&&|Dv^\e|_\infty\leq K,~~\text{where $K$ {\em is independent of} $\e,$}\label{grad-intro}\\
&&|Du(\cdot,t)|_\infty\leq K,~~\text{where $K$ {\em is independent of} $t$}\label{gradt-intro},
\end{eqnarray}
for strictly elliptic equations whose 
Hamiltonians $H$ have arbitrary growth power in the gradient variable,
which is unsual.
\smallskip

An important feature of our work is that we look for uniform gradient bounds
in $\e$ or $t.$ In many results, the bounds depend crucially on the $L^\infty$ norm
of the solution (which looks like $O(\e^{-1})$ or $O(t)$), something we want to avoid
in order to be able to solve some ergodic problems by sending $\e\to 0$
or to study the large time behavior of $u(x,t)$ when $t\to +\infty.$
These applications are discussed more in details below and are done
in Section~\ref{Applications}. We focus now on the more delicate part, i.e.,
the Lipschitz bounds for~\eqref{approx-scal}.
\smallskip

Let us start by recalling the existing results when $H$ is superquadratic
and coercive.
H\"older regularity of the
solution is proved under the very general assumption
\begin{eqnarray*}
H(x,p)\geq \frac{1}{C}|p|^k -C, \qquad
\text{with $k>2,$} 
\end{eqnarray*}
see Capuzzo Dolcetta et al.~\cite{clp10}, Barles~\cite{barles10},
Cardaliaguet-Silvestre~\cite{cs12}, Armstrong-Tran~\cite{at15}. 
But there are only few results as far as Lipschitz regularity
is concerned.
In general they are established using Bernstein method~\cite{gt98, lions82}
or the adaptation of this method in the context of viscosity solutions,
see Barles~\cite{barles91a}, Barles-Souganidis~\cite{bs01}, Lions-Souganidis~\cite{ls05},
Capuzzo Dolcetta et al.~\cite{clp10}.
This approach requires
some structural assumptions on $H$ which are often close to ``convexity-type
assumptions''. They appear naturally when differentiating the equation,
a drawback of the original Bernstein method. Even if
the weak Bernstein method~\cite{barles91a} is less restrictive as far as
the regularity of the datas is concerned (Lipschitz continuity
is enough), we do not consider this approach here to be able to deal
with Hamiltonians having few regularity like H\"older continuous Hamiltonians
for instance. Actually most of our assumptions do not even require the Hamiltonian
to be continuous as soon as a continuous solution to the equation exists.
However, let us mention that the weak Bernstein method has also several
advantages: the method may be used for degenerate equations in some cases
and the Hamiltonian may have arbitrary growth, see for instance~\cite{bs01, clp10}.
\smallskip

Instead, in this work, we use the Ishii-Lions' method introduced in~\cite{il90}, 
see also~\cite{cil92, barles08}.
This method allows to takes profit of the strict ellipticity of the equation to control
the strong nonlinearities of the Hamiltonian. 
In Ishii-Lions~\cite{il90} and Barles~\cite{barles91b},
weak regularity assumptions are assumed
over $H,$ merely a kind of balance between some H\"older continuity in $x$ and the
growth size of $H$ with respect to the gradient, namely
\begin{eqnarray}\label{il-assumpt}
&& |H(x,p)-H(y,p)|\leq \omega(|x-y|)|x-y|^\tau |p|^{2+\tau} + C
\quad \text{in \cite[Assumption (3.2)]{il90}},
\end{eqnarray}
or
\begin{eqnarray}\label{barles-assumpt}
&& |H(x,p)-H(y,p)|\leq C|x-y| |p|^{3} + C(1+|p|^2)
\quad \text{in \cite[Assumption (3.4)]{barles91b}},
\end{eqnarray}
where $x,y\in\T^N,$ $p\in\R^N,$ $\tau\in [0,1],$ $\omega$ is a modulus
of continuity and $C>0.$
These assumptions are designed for subquadratic (or growing at most like $|p|^3$)
Hamiltonians. This is not surprising since
it is known that, in general, the ellipticity is not powerful enough
to control nonlinearities which are more than quadratic~\cite{clp10}.
Under these assumptions, the authors prove a Lipschitz bound, which depends
however of the $L^\infty$ norm of the solution.
\smallskip

Our results consists in improving the previous ones in the periodic
setting. We give two new results, the first one being a slight generalization of 
of~\cite{il90, barles91b} while the second one takes profit of the strong
coercivity of $H$ and allows arbitrary growth of $H$ with respect to the gradient.
\begin{thm}\label{uni_grad1}
Assume \eqref{sig-deg} and  $H$ satisfies 
\begin{eqnarray}\label{ssa4}
&&  \left\{\begin{array}{l}
\text{there exists $L>1$ such that
for all $x,y\in\T^N,$}\\
\text{if $|p|\!= \!L$, then }
\displaystyle  H(x,p) \ge|p|\left[H(y,\frac{p}{|p|})\!+\!|H(\cdot ,0)|_\infty
\!+\! N|x-y||\sigma_x|_\infty^2\right].
\end{array}\right.
\end{eqnarray}
and
\begin{eqnarray}\label{LN-ell1}
&&  \left\{\begin{array}{l}
\text{There are constants $\alpha>0,~~C$ such that}\\[2mm]
\text{for all $x,y\in \T^N, p \in \R^N$, }\displaystyle 
|H(x,p)-H(y,p)| \le C|x-y|^{ \alpha}|p|^{\alpha+2}+C(1+|p|^{2}).
\end{array}\right.
\end{eqnarray}
Then, there exists $K>0$ such that for all $\e >0,$ any continuous solution 
$v^\e$ of \eqref{approx-scal} satisfies~\eqref{grad-intro}.
\end{thm}

\begin{thm}\label{uni_grad2}
Assume \eqref{sig-deg} and  $H$ satisfies 
\begin{eqnarray}\label{clp}
&& \text{there exist  constants $k>2,C>0$ such that } H(x,p)\ge \frac{1}{C}|p|^{ k}-C
\end{eqnarray}
and
\begin{eqnarray}\label{LN-ell2}
&&  \left\{\begin{array}{l}
\text{there exist a modulus of continuity $\omega$ and constants $\alpha \in [0,1], \beta<k-1$}\\[2mm]
\text{such that for all $x,y\in \T^N, p \in \R^N$,}\\[2mm]
\displaystyle 
|H(x,p)-H(y,p)| \le \omega \left((1+|p|^\beta)|x-y|\right)\, |x-y|^{ \alpha}|p|^{ (k-1)\alpha+k}
+o(|p|^k),
\end{array}\right.
\end{eqnarray}
where $o(|p|^k)/|p|^k\to 0$ as $|p|\to +\infty,$ uniformly with
respect to $x\in\T^N.$
Then, there exists $K>0$ such that for all $\e >0,$ any continuous solution 
$v^\e$ of~\eqref{approx-scal} satisfies~\eqref{grad-intro}.
\end{thm}
Before giving some comments about these results,
let us explain in a formal way the strategy to
establish them. The proof follows roughly the same lines
as the one in~\cite{barles91b}. We aim at proving that
the maximum
\begin{eqnarray*}
\mathop{\rm max}_{x,y\in\T^N} \{v^\e(x)-v^\e(y)-\psi(|x-y|) \}
\end{eqnarray*}
is nonnegative, choosing in a first step $\psi(r)=Lr^\alpha,$
$\alpha\in (0,1),$ to obtain a H\"older bound, and,
in a second step, $\psi(r)=L(r-r^{1+\alpha}),$ to improve
the H\"older bound into a Lipschitz one.
To do this, we use in a crucial way the strict concave
behavior of $\psi$ near 0 to take profit of the strict ellipticity
of the equation as usual in Ishii-Lions' method.
\smallskip

The first notable difference with the previous works is that
we are able to force the maximum to be achieved at
$(\overline{x},\overline{y})$ with $r:=|\overline{x}-\overline{y}|$
enough close to 0 without increasing $L$ in terms of
the $L^\infty$ norm of $v^\e.$ This is a consequence
of an a priori oscillation bound
\begin{eqnarray}\label{osc-intro}
{\rm osc}(v^\e):= \mathop{\rm sup}_{\T^N} v^\e
- \mathop{\rm inf}_{\T^N} v^\e \leq K, \quad
\text{where $K$ {\em is independent of} $\e,$}
\end{eqnarray}
obtained by the authors~\cite{ln16}
for any continuous
solution of~\eqref{approx-scal}
when merely~\eqref{ssa4} holds.
Let us underline that this oscillation bound is a crucial
tool in our work and that the assumption~\eqref{ssa4} is very general; 
it is satisfied as soon as
\begin{eqnarray}\label{Hsuperlin}
\mathop{\rm lim\,sup}_{|p|\to +\infty}\frac{H (x,p)}{|p|} =  +\infty
\text{ uniformly with respect to $x$.}
\end{eqnarray}
We extend the oscillation bound
in the parabolic setting, see Lemma~\ref{para oscillation},
and give an application.
\smallskip

The second step starts by noticing that, once we have on hands
a H\"older bound, then the strength of the nonlinearity
is weakened. We can apply again Ishii-Lions' method in a context
where the ellipticity is reinforced compared to the nonlinearity,
even when the Hamiltonian has a large growth with respect to the gradient.
It allows to improve the regularity up to Lipschitz continuity.
This is one of the main novelty to obtain the gradient bounds.
Then, a careful study of the balance between both terms finally gives
the best exponents.
\smallskip

Let us comment our results. Theorem~\ref{uni_grad1}
reduces to~\cite[III.1]{barles91b} when $\alpha=1.$
But notice that our Lipschitz bound does not depend on
the $L^\infty$ bound of the solution and
we are able to deal with Hamiltonians having less
regularity with respect to $x.$ For instance,
our result applies when
\begin{eqnarray}\label{exple-type-convex}
H(x,p)= |\Sigma(x)p|^m + G(x,p),
\quad\text{$m\leq \alpha+ 2,$} \ \Sigma\in C^{0,\alpha}(\T^N;\mathcal{M}_N),
\end{eqnarray}
and $G$ satisfies~\eqref{Hsuperlin} (superlinearity)
and $|G(x,p)|\leq C(1+|p|^2)$ (subquadratic) without any regularity condition
on $G.$
\smallskip

In Theorem~\ref{uni_grad2}, the coercivity assumption~\eqref{clp}
is the one needed to obtain the H\"older regularity
with exponent $\frac{k-2}{k-1}$ in~\cite{clp10}. Notice that,
this estimate being independent of~$\e$, we get for free the
oscillation bound~\eqref{osc-intro}. The first step in this case
consists in showing that the solution is $\gamma$-H\"older
continuous {\em for any} $\gamma \in (\frac{k-2}{k-1},1).$ 
It then allows us 
to improve the regularity up to Lipschitz continuity.
In~\eqref{LN-ell2}, the growth power with respect to the gradient variable  
can be much greater 
than $k>2,$ which enlarges the class of Hamiltonians under which
our result applies. Let us emphasize that the situation is very different comparing 
to Theorem~\ref{uni_grad1} where we can 
start with any H\"older exponent to get the Lipschitz regularity. Here, starting with 
a H\"older exponent equal to $\frac{k-2}{k-1}$ seems crucial to be able to improve the regularity
when $H$ has a strong growth with respect to the gradient.
\smallskip

As examples of applications of Theorem~\ref{uni_grad2}, 
we can deal with some new classes of Hamiltonians 
for which the existing regularity theory does not apply.
We can first consider again~\eqref{exple-type-convex}, where now
there exists $k>2$ such that
\begin{eqnarray*}
\text{$k \leq m\leq (k-1)\alpha +k,$ \ $\Sigma >0$
and }
 \frac{G(x,p)}{|p|^k}\mathop{\to}_{|p|\to +\infty} 0.
\end{eqnarray*}
Notice that even if $\Sigma$ is now
assumed to be nondegenerate,
this Hamiltonian is not necessarily convex.
\smallskip

The Hamiltonian
\begin{eqnarray*}
H(x,p)= a(x)h(p) + G(x,p),
\quad\text{$k>2$, \ $\frac{|p|^k}{C}\leq h(p)\leq C|p|^{ k},$ \
$\frac{G(x,p)}{|p|^k}\mathop{\to}_{|p|\to +\infty} 0,$}
\end{eqnarray*}
where $a$ is merely continuous and positive,
satisfies all the assumptions of Theorem~\ref{uni_grad2}
and is not convex in general. 
\smallskip

Let us give another example which will be used in 
Section~\ref{sec:parab} to extend the results to the
parabolic case~\eqref{edp-evol} and in Section~\ref{vanishing idea} 
to prove an existence result in a quite surprising situation.
Let $K$ be any continuous function satisfying 
$K(x,p)\le C(|p|^M+1)$, for any  $x\in\T^N, p \in \R^N, M>2.$
Then, the function
\begin{eqnarray*}
H(x,p)= K(x,p)+\alpha |p|^{M+\delta}, \quad \alpha>0,\delta>0
\end{eqnarray*}
satisfies all the assumptions of Theorem~\ref{uni_grad2}.
These examples also illustrate
the few regularity assumptions on the datas which are needed.
\smallskip

Our work takes place in the periodic setting to take profit
of the compactness and the absence of boundary of $\T^N.$
The issue of extending our results
in a bounded set is very interesting and not obvious.
In the case of Neumann boundary conditions, it should be true
but the case of Dirichlet boundary conditions
faces the problem of loss of boundary conditions
when $H$ is superquadratic~\cite{bdl04}. 
Notice that we cannot expect such general results to be true
in a general bounded set since it is known~\cite{clp10} that
$\frac{k-2}{k-1}$-H\"older continuity is optimal in general.
Our results can be 
extended for $A=\sigma\sigma^T$ with $\sigma\in C^{0,1/2}(\T^N;\mathcal{M}_N),$
for quasilinear equations when
$A=A(x,p)$ and for fully nonlinear equations
of Bellman-Isaacs type,
see Section~\ref{ext-quasilin} for a discussion.
\smallskip

To study the well-posedness of~\eqref{approx-scal}
under the assumptions of Theorems~\ref{uni_grad1} and~\ref{uni_grad2},
we have first to prove a comparison principle (Theorem~\ref{CP})
whose proof is not classical
since the Hamiltonian is not
Lipschitz continuous with respect to the gradient.
Instead, we use the same ideas as for the proof of the Lipschitz bounds. 
As a consequence, we obtain the existence
and uniqueness of a continuous viscosity solution
to~\eqref{approx-scal}. Moreover this solution is Lipschitz
continuous and, if the datas are $C^\infty$, then the solution
is $C^\infty$ thanks to the classical elliptic regularity theory.
Let us mention that our approach also allows to
construct H\"older continuous solutions to~\eqref{approx-scal}
(Theorem~\ref{existence})
under the general assumption~\eqref{poly H} which is not sufficient
to provide a comparison principle.
\smallskip

We then give several applications of our results.
A straightforward consequence to the bound~\eqref{grad-intro}
is the solvability of the ergodic problem
associated with~\eqref{approx-scal}, see \cite{lpv86, al98} and Theorem~\ref{thm-erg}:
there exists $(c,v^0)\in\R\times W^{1,\infty}(\T^N)$ solution to
\begin{eqnarray}\label{erg-intro}
-{\rm trace}(A(x) D^2v^0) +  H(x, Dv^0)=c, & x\in\T^N.
\end{eqnarray}

The next application is
the study of the parabolic equation~\eqref{edp-evol}.
The natural idea to extend the gradient bound for~\eqref{approx-scal} 
to~\eqref{edp-evol}
is to prove first a bound for the time derivative
$|\frac{\partial u}{\partial t}|_\infty$
and then to apply the results obtained for the stationary equation.
This approach does not work directly for several
reasons. On the one side, the bound for the time derivative is
usually obtained as a consequence of the comparison principle
which is not available here. On the other
side, our a priori stationary gradient bounds are valid for continuous
solutions and not for subsolutions. We overcome these difficulties
by considering a tricky approximate equation where $H$ is replaced by 
\begin{eqnarray*}
H_{nq}(x,p)=\frac{1}{q}|p|^M + H_n(x,p),
\end{eqnarray*}
with a bounded uniformly continuous approximation $H_n$ of $H.$
A crucial point is that,
since the coercive term $\frac{1}{q}|p|^M$ does not depend on $x,$
the comparison principle holds for this new equation allowing us to build a continuous
viscosity solution.
Moreover, the approximate Hamiltonian satisfies
the key assumptions~\eqref{ssa4}--\eqref{LN-ell1}
or~\eqref{clp}--\eqref{LN-ell2} 
with the same constants as the original $H$. So we can build a solution of the ergodic problem
in this case. This solution allows us to control
the $L^\infty$, oscillation and time derivative bounds of the solution of the
parabolic problem. We therefore can prove a parabolic version of
the H\"older regularity result of~\cite{clp10} using the strong coercivity
of $H_{nq}$ (Lemma~\ref{clp-time}). By this way, we are in position to mimic the proofs
of gradient bounds in the stationary case and to conclude to the existence
of a unique Lipschitz continuous
solution to~\eqref{edp-evol}, see Theorem~\ref{thm-evol-unif}.
\smallskip

We finally apply all the previous results to prove the large time behavior
of the solution of~\eqref{edp-evol}. Having on hands the gradient bound~\eqref{gradt-intro}, 
a solution of the ergodic problem~\eqref{erg-intro} and the strong maximum principle, 
the proof is classical~\cite{bs01}.
\smallskip

The paper is organized as follows. In 
Section~\ref{os gra}, we prove the stationary gradient bounds, 
Theorems~\ref{uni_grad1} and~\ref{uni_grad2}.
Section~\ref{NCP} is devoted to establish the well-posedness
of~\eqref{approx-scal}. Finally, the applications are presented
in Section~\ref{Applications}. We start by solving the ergodic problem,
then a study of the parabolic equation~\eqref{edp-evol} is provided.
We end with the long-time behavior of the solution of~\eqref{edp-evol}
and the construction of H\"older continuous solutions to equations
with Hamiltonians of arbitrary growth without the use of comparison
principle.
\smallskip

\noindent
{\bf Acknowledgement.} 
This work was partially supported by the ANR (Agence Nationale de
la Recherche) through HJnet project ANR-12-BS01-0008-01
and WKBHJ project ANR-12-BS01-0020.

\section{Gradient bound for the stationary equation~~\eqref{approx-scal}}\label{os gra}

\subsection{Oscillation bound}

\begin{lem}\label{oscillation}
Assume~\eqref{ssa4}. Let  $v^\e$ be a continuous
solution of~\eqref{approx-scal} and let $v^\e(x_\e)=\min v^\e$. Then
\begin{eqnarray*}
v^\e (x)- v^\e(x_\e)\le L|x-x_\e| \quad \text{for all $x\in \T^N$,}
\end{eqnarray*}
where $L$ is the constant (independent of $\e$) which appears in~\eqref{ssa4}.
\end{lem}
An immediate consequence is
\begin{eqnarray*}
{\rm osc}(v^\e):=  \max v^\e-\min v^\e\leq \sqrt{N}L.
\end{eqnarray*}
To make the article self-contained, we present the proof of this result in Appendix.

\subsection{Preliminary lemma for Ishii-Lions's method}

The following technical lemma is a key tool in this article.
\begin{lem}\label{tech lemma}
Supose $v^\e$ is a continuous viscosity solution of~\eqref{approx-scal}
in some open subset $\Omega$ with $A(x)=\sigma(x)\sigma^T(x),$
$\sigma\in W^{1,\infty}(\overline{\Omega}).$ 
Let $\Psi:\R^+ \to \R^+$ be an increasing concave 
function such that $\Psi(0)=0$ and the maximum of  
\begin{eqnarray*}
\max_{x,y \in \overline{\Omega}}\{ v^\e(x)-v^\e(y)-\Psi(|x-y|)\},
\end{eqnarray*}
is achieved at $(\overline{x},\overline{y}).$
If we can write the viscosity inequalities for $v^\e$
at $\overline{x}$ and $\overline{y},$ then 
for every
$\varrho>0,$ there exists
$(p,X) \in \overline{J}^{2,+}v^\e(\overline{x}),(p,Y) \in \overline{J}^{2,-}v^\e(\overline{y})$
such that
\begin{eqnarray}\label{mat}
\left(
\begin{array}{ccc}
X & 0 \\
0 & -Y
\end{array}
\right)
\le A+\varrho A^2,
\end{eqnarray}
with
\begin{eqnarray}
\label{mat-bis}
p=\Psi'(|\overline{x}-\overline{y}|) q,
\quad q=\frac{\overline{x}-\overline{y}}{|\overline{x}-\overline{y}|},
\quad B=\frac{1}{|\overline{x}-\overline{y}|} (I-q \otimes q),\\
\label{mat-ter}
A=\Psi'(|\overline{x}-\overline{y}|)
\left(
\begin{array}{ccc}
B & -B \\
-B & B
\end{array}
\right)
+\Psi''(|\overline{x}-\overline{y}|)
\left(
\begin{array}{ccc}
q \otimes q & -q \otimes q \\
-q \otimes q & q \otimes q
\end{array}
\right)
\end{eqnarray}
and the following estimate holds
\begin{eqnarray}\label{estim-trace1}
&&-{\rm trace}(A(\overline{x})X
-A(\overline{y})Y)\geq 
- N |\sigma_x|_\infty^2 |\overline{x}-\overline{y}|\Psi'(|\overline{x}-\overline{y}|)+O(\varrho).
\end{eqnarray}
If, in addition, \eqref{sig-deg} holds, then
there exists $\tilde{C}=\tilde{C}(N,\nu,|\sigma|_\infty, |\sigma_x|_\infty)$ (given by~\eqref{def-ctilde})
such that
\begin{eqnarray}\label{ineq-tracet}
-{\rm trace}(A(\overline{x})X-A(\overline{y})Y)
\geq
-4\nu\Psi''(|\overline{x}-\overline{y}|)
-\tilde{C}\Psi'(|\overline{x}-\overline{y}|)|\overline{x}-\overline{y}|
+O(\varrho)
\end{eqnarray}
and, if the maximum is positive, then
\begin{eqnarray} \label{estimation-outil}
&&-4\nu\Psi''(|\overline{x}-\overline{y}|)
-\tilde{C}\Psi'(|\overline{x}-\overline{y}|)|\overline{x}-\overline{y}|
+H(\overline{x},\Psi'(|\overline{x}-\overline{y}|)q)
-H(\overline{y},\Psi'(|\overline{x}-\overline{y}|)q) < 0.
\end{eqnarray}
\end{lem}

The first part of the result is a basic application of Ishii's Lemma
in viscosity theory, see~\cite{cil92}. The trace estimates can be found
in~\cite{il90, barles91b, bs01} and~\eqref{estimation-outil} 
takes benefit of the ellipticity of the equation and allows to
apply Ishii-Lion's method introduced in~\cite{il90}.
For reader's convenience, we provide
a proof in the Appendix.

\subsection{Proof of Theorem~\ref{uni_grad1}.}\label{Case 1}
The proof relies on some ideas of~\cite{barles91b}. The main difference is that,
thanks to the uniform oscillation bound presented in Lemma~\ref{oscillation},
we can obtain a gradient bound independent of the $L^\infty$ norm of the
solution.
\smallskip

\noindent{\it Step 1. H\"older continuity.}
We claim that there exist some constants $\gamma\in (0,1], K>0$ independent of 
$\e$ such that
\begin{eqnarray*}
|v^\e|_{C^{0,\gamma}} \le K_0.
\end{eqnarray*}
We skip the $\e$ superscript in $v^\e$  hereafter for sake of notations. 
 
Thanks to Lemma~\ref{oscillation}, the oscillation of $v$ is uniformly (in $\e$) 
bounded by a constant $\mathcal{O}$. 

Consider
\begin{eqnarray*}
\max_{x,y \in \T^N}\{ v(x)-v(y)-\Psi(|x-y|)\},
\end{eqnarray*}
where $\Psi(s)=K_0 s^{\gamma}$. Our goal is to choose $\gamma\in (0,1], K_0>0$, which depend only 
on  $C,\alpha$ given by the hypothesis~\eqref{LN-ell1} such that the above maximum 
is nonegative. To do so, we assume by contradiction that the maximum is positive and hence, 
it is achieved at $(\overline{x},\overline{y})$
with $\overline{x}\not= \overline{y}$ thanks to the continuity of $v$. We next choose 
$r$ depending on $K_0$ such that $K_0 r^{\gamma}=\mathcal{O}+1$.

With such a choice of $r$, it is clear that $|\overline{x}-\overline{y}| < r$. 
Denote $s:=|\overline{x}-\overline{y}|.$  
From Lemma~\ref{tech lemma} and~\eqref{LN-ell1}, 
we will have a contradiction if we can choose $K,\gamma$ such that
\begin{eqnarray*}
-4\nu \Psi''(s)-\tilde{C} s \Psi'(s) \geq Cs^{ \alpha}\Psi'(s)^{ \alpha+2}+C\Psi'(s)^2+C.
\end{eqnarray*}
Computing $\Psi'(s)= K_0\gamma s^{\gamma-1}$ and  $\Psi''(s)= K_0\gamma (\gamma -1) s^{\gamma-2},$
we have to prove
\begin{eqnarray*}
4\nu K_0\gamma(1-\gamma)s^{ \gamma-2}-\tilde{C}K_0\gamma s^{\gamma} \ge 
Cs^{ \alpha}(K_0\gamma s^{\gamma-1})^{ \alpha+2}+
C(K_0\gamma s^{\gamma -1})^2 +C.
\end{eqnarray*}

It is clear that $\nu K_0\gamma(1-\gamma)s^{ \gamma-2}\ge \tilde{C}K_0\gamma s^{\gamma}+C$ 
when $r$ is small enough. Hence, the above inequality holds true if we can choose 
$K_0,\gamma$ such that the two following inequalities hold,
\begin{eqnarray*}
\nu K_0\gamma(1-\gamma)s^{ \gamma-2} \ge Cs^{ \alpha}(K_0\gamma s^{\gamma-1})^{ \alpha+2}
&\Leftrightarrow& \nu (1-\gamma) \ge C(K_0s^{\gamma})^{ \alpha+1}\gamma^{ \alpha+1},
\end{eqnarray*}
and 
\begin{eqnarray*}
\nu K_0\gamma(1-\gamma)s^{ \gamma-2} \ge C\Psi'^2 \Leftrightarrow \nu (1-\gamma)\ge C K_0s^{\gamma}\gamma.
\end{eqnarray*}
Since $K_0 s^{\gamma}\leq K_0 r^{\gamma}\leq \mathcal{O}+1,$ both inequalities hold true
when $\gamma$ is small enough depending on the oscillation $\mathcal{O}$ (but not on $K_0$).
This proves the claim.
\smallskip

\noindent{\it Step 2. Improvement of the H\"older regularity to Lispchitz
regularity.} From the previous step,
$v$ is $\gamma$  H\"older continuous  ($\gamma$ is possibly small)
and the H\"older constant $K_0$ can be chosen to be independent of $\e$. 
We fix such a $\gamma$. We also recall that, from Lemma~\ref{oscillation},  
the oscillation of $v$ is bounded by a constant $\mathcal{O}$ independent of $\e.$ 

We first construct a concave function $\Psi: [0,r]\to \R_+$ by
\begin{eqnarray}\label{fct-conc-lip-hold}
\Psi(s)=A_1 [A_2s-(A_2s)^{1+\gamma}],
\end{eqnarray}
where $r, A_1, A_2>0$, which depend only on  $C,\alpha,\beta$ given by the hypothesis~\eqref{LN-ell1}, 
will be precised later. We extend $\Psi$ into $\R_+$
by defining $\Psi(s)=\Psi(r)$ for $s \ge r$.

We compute, for $0\leq s<r,$
\begin{eqnarray*}
\Psi'(s)=A_1A_2 [1-A_2^{\gamma}(1+\gamma)s^{\gamma}],
\quad \Psi''(s)=-A_1A_2^{1+\gamma}\gamma(1+\gamma)s^{\gamma-1}<0.
\end{eqnarray*}

We then choose $r$ depending on $A_2$ ($A_2$ may vary in the next arguments) 
such that
\begin{eqnarray}\label{R1}
A_2r=\frac{1}{3} \quad\text{and} \quad \Psi(r)=\mathcal{O}+1
\end{eqnarray}
A consequence of this choice is that $A_1$ is now fixed
since 
$A_1 (3^{-1}-3^{-1-\gamma})=\mathcal{O}+1.$
It is straightforward to see that $\Psi$ is a smooth concave increasing 
function on $[0,r)$ satisfying $\Psi(0)=0$ and, for all $s\in [0,r],$
\begin{eqnarray}\label{ineg-grad}
A_1A_2 [1-\frac{1+\gamma}{3^\gamma}] = \Psi'(r)\leq \Psi'(s)\leq \Psi'(0)=A_1A_2.
\end{eqnarray}

Consider
\begin{eqnarray*}
M:=\max_{x,y \in \T^N}\{ v(x)-v(y)-\Psi(|x-y|)\}.
\end{eqnarray*}
If $M\le 0$ then the theorem holds with $K=A_1A_2.$
The rest of the proof consists in proving that $M$ is indeed nonpositive
for $A_2$ big enough.
We argue by contradiction assuming that $M>0.$ This maximum is achieved at 
$(\overline{x},\overline{y})$
with $\overline{x}\not= \overline{y}$. With the choice of $r$ in the condition~\eqref{R1} 
and the fact that $\Psi$ is non-decreasing, it is clear that $|\overline{x}-\overline{y}| < r$.
 
Denote $s:=|\overline{x}-\overline{y}|$.
From~~\eqref{LN-ell1} and Lemma~~\ref{tech lemma},
we have
\begin{eqnarray*}
&&-4\nu  \Psi''(s) -\tilde{C}s \Psi'(s)< Cs^{\alpha}\Psi'(s)^{\alpha+2}+C+C\Psi'(s)^2,
\end{eqnarray*}
which gives us
\begin{eqnarray*}
 4\nu  A_1 A_2^{1+\gamma}\gamma(1+\gamma)s^{\gamma-1} 
-\tilde{C}A_1A_2 s[1-A_2^{\gamma}(1+\gamma)s^{\gamma}]<Cs^{\alpha}\Psi'(s)^{\alpha+2}+C+C\Psi'(s)^2.
\end{eqnarray*}
The goal now is to have a contradicton in the above inequality for large $A_2.$

We first note that it is possible to increase $A_2$ in order that
\begin{eqnarray}\label{R2}
\nu  A_1 A_2^{1+\gamma}\gamma(1+\gamma)s^{\gamma-1} -\tilde{C}A_1A_2 s[1-A_2^{\gamma}(1+\gamma)s^{\gamma}]\geq 0.
\end{eqnarray}
Indeed, the inequality is true for all $A_2\geq 1$ if $s\geq 1$ and, when $s\leq 1,$ it is sufficient
to take $A_2\geq (\nu\gamma)^{-1}\tilde{C}.$

Therefore, it is enough to show that we may choose $A_2$ such that
the following inequalities hold true,
\begin{eqnarray}\label{hope1}
\nu  A_1 A_2^{1+\gamma}\gamma(1+\gamma)s^{\gamma-1} \ge Cs^{\alpha}\Psi'(s)^{\alpha+2}
=C\big (s\Psi'(s)^{ \frac{1}{1-\gamma}}\big )^{ \alpha}~~\Psi'(s)^{ \alpha+2-\frac{\alpha}{1-\gamma}},
\end{eqnarray}
and
\begin{eqnarray}\label{hope2}
2\nu  A_1A_2^{1+\gamma}\gamma(1+\gamma)s^{\gamma-1}\ge C+C\Psi'(s)^2.
\end{eqnarray}

We first prove that it is possible to choose 
$A_2$ such that \eqref{hope1} holds true.
We know that $\Psi$ is concave and $\gamma$-H\"older continuous, 
so we have 
$$
s\Psi'(s) \le \Psi(s) < v(\overline{x})-v(\overline{y})\le K_0 s^{ \gamma},
$$
Hence 
\begin{eqnarray}\label{interm121}
s \Psi'(s)^{ \frac{1}{1-\gamma}} \le K_0^{ \frac{1}{1-\gamma}}
\end{eqnarray}
and it follows that~\eqref{hope1} is true provided
\begin{eqnarray}\label{interm122}
\nu  A_1A_2^{1+\gamma}\gamma(1+\gamma)s^{\gamma-1} 
\ge CK_0^{ \frac{\alpha}{1-\gamma}}~(A_1 A_2)^{ \alpha+2-\frac{\alpha}{1-\gamma}}.
\end{eqnarray}
Recalling that $1/s \ge 1/r > A_2$ from~\eqref{R1}, we have
\begin{eqnarray}\label{interm123}
A_2^{1+\gamma}s^{\gamma-1} \geq A_2^2
\end{eqnarray}
and~\eqref{interm122} is true if
\begin{eqnarray}\label{R3}
\nu  A_1A_2^2\gamma(1+\gamma)\ge CK_0^{ \frac{\alpha}{1-\gamma}}~~(A_1 A_2)^{ \alpha+2-\frac{\alpha}{1-\gamma}}.
\end{eqnarray}
Finally,~\eqref{R3} indeed holds for $A_2$ big enough since $\alpha+2-\frac{\alpha}{1-\gamma}<2$.

We now prove that it is possible to choose 
$A_2$ such that \eqref{hope2} holds true.
At first, from~\eqref{interm123}, we have
$\nu  A_1A_2^{1+\gamma}\gamma(1+\gamma)s^{\gamma-1}\ge \nu  A_1A_2^2 \gamma(1+\gamma)\geq C$ 
when $A_2$ big enough. So, \eqref{hope2} holds if we can choose $A_2$ such that 
\begin{eqnarray}\label{ineg134}
\nu  A_1A_2^{1+\gamma}\gamma(1+\gamma)\ge Cs^{1-\gamma}\Psi'(s)^2.
\end{eqnarray}
From~\eqref{ineg-grad} and~\eqref{interm121},
\begin{eqnarray*}
Cs^{1-\gamma}\Psi'(s)^2=C(s\Psi'(s)^{\frac{1}{1-\gamma}})^{1-\gamma}\Psi'(s)
\leq C K_0 A_1 A_2,
\end{eqnarray*}
so~\eqref{ineg134} holds provided
\begin{eqnarray*}
\nu  A_1A_2^{1+\gamma}\gamma(1+\gamma)\ge C K_0 A_1A_2,
\end{eqnarray*}
which is obviously true if $A_2$ is big enough.

The proof of the theorem is complete
\hfill$\Box$
\subsection{Proof of Theorem \ref{uni_grad2}.}\label{Case 2}
From~\cite{clp10}, we have
\begin{eqnarray}\label{clp2010}
v \text{ is $\frac{k-2}{k-1}$-H\"older continuous and the H\"older constant is equal to $K_0$},
\end{eqnarray}
where $k>2$ is given by the  assumption~\eqref{clp}. 
In~\cite{clp10},  the authors prove that the H\"older constant depends only on 
$N,k,|\e v^\e|_\infty$ and, since $|\e v^\e|_\infty \le |H(x,0)|_\infty$, 
$K_0$ can be chosen independent of $\e$.
A by-product of the above result (or of Lemma~\ref{oscillation}) is that 
the oscillation of $v^\e$ is bounded by a constant $\mathcal{O}>0$
independent of $\e.$ 

Hereafter we write $v$ for $v^\e.$
\smallskip

\noindent{\it Step 1. Improvement of the H\"older exponent.}
Fix any $\chi\in (\frac{k-2}{k-1},1)$. We show that $v$ is $\chi$-H\"older continuous.
 
We set
\begin{eqnarray*}
&&\Psi(s)=K s^{\chi},
\end{eqnarray*}
where $K>0$, which depend only on  $C,\alpha,\beta$ given by the hypothesis~\eqref{LN-ell2}, 
will be precised later. We fix a constant $r$ which depends on $K$ as follows
\begin{eqnarray}\label{fixrad}
Kr^{\chi}=\mathcal{O}+1.
\end{eqnarray}
Consider
\begin{eqnarray}\label{max135}
\max_{x,y \in \T^N}\{ v(x)-v(y)-K|x-y|^{\chi}\}.
\end{eqnarray}
If the maximum is nonpositive then the theorem holds.
From now on, we argue by contradiction
assuming that the maximum is positive. The maximum is achieved at 
$(\overline{x},\overline{y})$ with $\overline{x}\not= \overline{y}$. 
With the choice of $r$ in~\eqref{fixrad}, it is clear 
that $|\overline{x}-\overline{y}| < r$.

Denote $s:=|\overline{x}-\overline{y}|$.
From~\eqref{LN-ell2} and Lemma~\ref{tech lemma}, we have
\begin{eqnarray*}
-4\nu  \Psi''(s) -\tilde{C}s \Psi'(s)< \omega\left( (1+\Psi'(s)^\beta)s\right)\,s^{\alpha}\Psi'(s)^{(k-1)\alpha+k}+o(|p|^k),
\end{eqnarray*}
with $\beta<k-1,$
$\Psi'(s)=K\chi s^{\chi-1},$ $\Psi''(s)=K\chi(\chi-1)s^{\chi-2}$
and $p=\Psi'(s)\frac{\overline{x}-\overline{y}}{|\overline{x}-\overline{y}|}.$ 
We can rewrite the above inequality as
\begin{eqnarray*}
4\nu K\chi(1-\chi)s^{ \chi-2} -\tilde{C}K\chi s^{ \chi}<\omega\left((1+\Psi'(s)^\beta)s\right)\, 
s^{\alpha}\Psi'(s)^{(k-1)\alpha+k}+o(|p|^k).
\end{eqnarray*}
At first, from~\eqref{fixrad}, it is possible to increase $K$
such that $r$ is small enough in order to have
\begin{eqnarray*}
2\nu K\chi(1-\chi)s^{ \chi-2} \ge \tilde{C}K\chi s^{ \chi}, \quad \text{for $s\leq r.$}
\end{eqnarray*} 
Hence, to get a contradiction in the above inequality, we only need to choose $K$ such that
the two following inequalities hold,
\begin{eqnarray}\label{hol0}
\nu K\chi(1-\chi)s^{ \chi-2} \ge \omega\left((1+\Psi'(s)^\beta)s\right) s^{ \alpha}\Psi'(s)^{(k-1)\alpha+k}
\end{eqnarray}
and
\begin{eqnarray}\label{hol1}
\nu K\chi(1-\chi)s^{ \chi-2}\ge o(|p|^k).
\end{eqnarray}

\noindent{\it Step 1.1. Choosing $K$ large enough such that we have~\eqref{hol0}.} 
Writing that the maximum~\eqref{max135} is positive and using the
concavity of $\Psi$ and~\eqref{clp2010}, 
we have $s\Psi'(s) \le Ks^{ \chi}=\Psi(s) < v(\overline{x})-v(\overline{y})\le K_0 s^{ \frac{k-2}{k-1}}$, 
hence 
\begin{eqnarray}\label{strange2}
s\Psi'(s)^{ k-1} \le K_0^{ k-1}\text{ and }\frac{1}{s}\ge \left(\frac{K}{K_0}\right)^{ \frac{1}{\chi-\frac{k-2}{k-1}}}.
\end{eqnarray}
It follows
\begin{eqnarray*}
\omega\left((1+\Psi'(s)^\beta)s\right) s^{\alpha}\Psi'(s)^{(k-1)\alpha+k}
&=&\omega\left((1+\Psi'(s)^\beta)s\right) \left( s\Psi'(s)^{k-1}\right)^{\alpha} \Psi'(s)^{ k}\\
&\leq& \omega\left((1+\Psi'(s)^\beta)s\right) K_0^{\alpha(k-1)}(K\chi s^{\chi-1})^{k}.
\end{eqnarray*}
Therefore,~\eqref{hol0} is true provided
\begin{eqnarray*}
\nu K\chi(1-\chi)s^{\chi-2}
\ge 
\omega\left((1+\Psi'(s)^\beta)s\right) K_0^{\alpha(k-1)}(K\chi s^{ \chi-1})^{ k}.
\end{eqnarray*}
Setting $\tilde{\nu}= \nu \chi(1-\chi)K_0^{\alpha(1-k)}\chi ^{-k},$
which is a constant independent of $K, s,$ we rewrite the above desired inequality as
\begin{eqnarray}\label{hol5}
&& 
\tilde{\nu}\left(\frac{1}{s}\right)^{ k(\chi-1)-\chi+2}\ge \omega\left((1+\Psi'(s)^\beta)s\right) K^{k-1}.
\end{eqnarray}
From~\eqref{strange2} and the choice $\chi> \frac{k-2}{k-1},$
it follows that inequality~\eqref{hol5} holds true if 
\begin{eqnarray*}
\tilde{\nu} \left(\frac{K}{K_0}\right)^{  \frac{k(\chi-1)-\chi+2}{\chi-\frac{k-2}{k-1}}}
\ge  \omega\left((1+\Psi'(s)^\beta)s\right) K^{ k-1}
\end{eqnarray*}
or 
\begin{eqnarray}\label{ineq701}
\tilde{\nu} \left(\frac{1}{K_0}\right)^{ k-1}\ge  \omega\left((1+\Psi'(s)^\beta)s\right).
\end{eqnarray}
Using~\eqref{strange2} and $\beta <k-1,$ we have
\begin{eqnarray*}
s\Psi'(s)^\beta= s^{1-\frac{\beta}{k-1}}\left(s\Psi'(s)^{k-1}\right)^{\frac{\beta}{k-1}}
\leq s^{\frac{k-1-\beta}{\beta}}K_0^\beta \leq r^{\frac{k-1-\beta}{\beta}}K_0^\beta \mathop{\to}_{r\to 0} 0.
\end{eqnarray*}
Finally~\eqref{ineq701} holds true for large $K$ since $r\to 0$ as $K\to +\infty$
by~\eqref{fixrad}. This proves~\eqref{hol0}. 

\smallskip

\noindent{\it Step 1.2. Choosing $K$ large enough such that we have~\eqref{hol1}.} 
We have
\begin{eqnarray*}
o(|p|^k)=\frac{o(|p|^k)}{\Psi'(s)^k}(K\chi s^{ \chi-1})^{k}
=\frac{o(|p|^k)}{\Psi'(s)^k}\chi^k \Psi(s)^{k-1},
\end{eqnarray*}
so~\eqref{hol1} holds if
\begin{eqnarray*}
 \nu (1-\chi) \ge \frac{o(|p|^k)}{\Psi'(s)^k}\chi^{k-1}s^{2-k}\Psi(s)^{k-1}.
\end{eqnarray*}
Recalling that $\Psi(s)\leq K_0 s^{\frac{k-2}{k-1}},$
it is sufficient to ensure
\begin{eqnarray*}
\nu (1-\chi) \ge \frac{o(|p|^k)}{\Psi'(s)^k} (\chi K_0)^{k-1}.
\end{eqnarray*}
Since 
\begin{eqnarray*}
|p|=\Psi'(s)=\chi K s^{\chi-1}\geq \chi K r^{\chi-1}= \chi (\mathcal{O}+1)^{\frac{\chi -1}{\chi}}K^{\frac{1}{\chi}}
\end{eqnarray*}
by~\eqref{fixrad}, we have that $|p|\to +\infty$ as $K\to +\infty.$
We then obtain that the above inequality holds true for large $K$
concluding~\eqref{hol1}. This ends Step 1.
\smallskip

\noindent{\it Step 2.  Improvement of the new H\"older exponent to Lipschitz continuity.}
We are now ready to prove the lipschitz continuity.

The beginning of the proof is similar to the one of Theorem~\ref{uni_grad1}.
We consider the increasing concave function $\Psi$ given by~\eqref{fct-conc-lip-hold}
for any $\gamma\in (0,1)$ and $A_1, A_2,r>0$ satisfying~\eqref{R1} and set
\begin{eqnarray*}
M=\max_{x,y \in \T^N}\{ v(x)-v(y)-\Psi(|x-y|)\}.
\end{eqnarray*}
We are done if the maximum
is nonnegative.
Assuming by contradiction that the maximum is positive,
we know it is achieved at $(\overline{x},\overline{y})$ with $s:=|\overline{x}-\overline{y}| < r$.
Applying Lemma~\ref{tech lemma} and~\eqref{LN-ell2}, we see that we reach
the desired contradiction if the following inequalities hold
\begin{eqnarray}\label{BLN1}
\nu  A_1 A_2^{1+\gamma}\gamma(1+\gamma)s^{\gamma-1} 
&\ge& \omega\left((1+\Psi'(s)^\beta)s\right) s^{\alpha}\Psi'(s)^{(k-1)\alpha+k}\\
&=&\omega\left((1+\Psi'(s)^\beta)s\right) (s\Psi'(s)^{\frac{1}{1-\chi}})^{ \alpha} \Psi'(s)^{ (k-1)\alpha+k-\frac{\alpha}{1-\chi}}\nonumber
\end{eqnarray}
and 
\begin{eqnarray}\label{BLN2}
\nu  A_1 A_2^{1+\gamma}\gamma(1+\gamma)s^{\gamma-1} \ge o(|p|^k)
\quad \text{where $|p|=\Psi'(s)$}.
\end{eqnarray}
Next substeps are devoted to prove that we can fulfill the two above inequalities
by choosing $A_2$ large enough. It then leads to a contradiction which implies
that the maximum $M$ is nonnegative concluding that $v$ is Lipschitz continuous
with constant $A_1A_2$ and ending the proof of Theorem~\ref{uni_grad2}.
\smallskip

\noindent{\it Step 2.1. Choosing $A_2$ such that \eqref{BLN1} holds true.}
From Step 1, we know that $v$ is $\chi$-H\"older continuous for any $\chi \in (\frac{k-2}{k-1},1)$
with a constant $K=K_\chi$ which is independent of~$\e.$
We then have $s\Psi'(s) \le \Psi(s) < K_\chi s^\chi$, hence 
\begin{eqnarray}\label{strange10}
s\Psi'(s)^{ \frac{1}{1-\chi}} \le K_\chi^{ \frac{1}{1-\chi}}.
\end{eqnarray}
Moreover
\begin{eqnarray}
\label{strange1010}
&& c_\gamma A_1A_2:= \frac{{\rm log}\,3-1}{3}\gamma  A_1A_2  
\leq A_1A_2 \left( 1-\frac{1+\gamma}{3^\gamma}\right) = \Psi'(r)\leq \Psi'(s)\leq \Psi'(0)=A_1A_2.
\end{eqnarray}
It follows that~\eqref{BLN1} holds provided
\begin{eqnarray*}
&&\nu  A_1 A_2^{1+\gamma}\gamma(1+\gamma)s^{\gamma-1} \ge \omega\left((1+\Psi'(s)^\beta)s\right)
K_\chi^{ \frac{\alpha}{1-\chi}} c_\gamma^{-|(k-1)\alpha+k-\frac{\alpha}{1-\chi}|}
(A_1 A_2)^{ (k-1)\alpha+k-\frac{\alpha}{1-\chi}}.
\end{eqnarray*}
Recalling that $1/s \ge 1/r > A_2$,  the above inequality is true if
\begin{eqnarray*}
&&\nu  A_1 A_2^2\gamma(1+\gamma) \ge \omega\left((1+\Psi'(s)^\beta)s\right) K_\chi^{ \frac{\alpha}{1-\chi}} 
c_\gamma^{-|(k-1)\alpha+k-\frac{\alpha}{1-\chi}|}(A_1 A_2)^{ (k-1)\alpha+k-\frac{\alpha}{1-\chi}}.
\end{eqnarray*}
First of all, we have $\beta<k-1<\frac{1}{1-\chi},$
so, by~\eqref{BLN1},  $\omega\left((1+\Psi'(s)^\beta)s\right)$ is small for small $s.$  
Therefore, to fulfill~\eqref{BLN1}, it is enough to fix
$\chi$ close enough to 1 such that 
\begin{eqnarray}\label{choix-chi123}
(k-1)\alpha+k-\frac{\alpha}{1-\chi}<2
\quad\Leftrightarrow\quad \chi>1-\frac{\alpha}{(k-1)\alpha+k-2}
\end{eqnarray}
and to take $A_2$ large enough.
\smallskip

\noindent{\it Step 2.2. Choosing $A_2$ such that~\eqref{BLN2} holds true.} 
We need to choose $A_2$ such that
\begin{eqnarray*}
\nu  A_1 A_2^{1+\gamma}\gamma(1+\gamma) &\ge& s^{1-\gamma}o(|p|^k)\\
&=&\frac{o(|p|^k)}{\Psi'(s)^k} s^{1-\gamma}\Psi'(s)^k=\frac{o(|p|^k)}{\Psi'(s)^k} (s\Psi'(s)^{\frac{1}{1-\chi}})^{ 1-\gamma} 
\Psi'(s)^{ k-\frac{1-\gamma}{1-\chi}}.
\end{eqnarray*}
Using~\eqref{strange10} and~\eqref{strange1010} again,
we see that the above inequality is true provided
\begin{eqnarray*}
\nu  A_1 A_2^{1+\gamma}\gamma(1+\gamma) 
\ge \frac{o(|p|^k)}{\Psi'(s)^k}K_\chi^{\frac{1-\gamma}{1-\chi}} (A_1A_2)^{ k-\frac{1-\gamma}{1-\chi}}.
\end{eqnarray*}
We fix $\chi\in (\frac{k-2}{k-1},1)$ close enough to 1 such that~\eqref{choix-chi123} holds and
$k-\frac{1-\gamma}{1-\chi}<1+\gamma$. 
Noticing that $\Psi'(s)=|p|\to +\infty$ when $A_2\to +\infty,$ 
the previous inequality holds
when $A_2$ is big enough. Therefore~\eqref{BLN2} holds.
The proof of the theorem is complete.
\hfill$\Box$

\subsection{Extensions}
\label{ext-quasilin} 

As said in the introduction, all proofs still hold true when when $\sigma \in C^{0,\theta}(\T^N;\mathcal{M}_N),$
$\theta\in [\frac{1}{2},1],$
instead of $W^{1,\infty}(\T^N;\mathcal{M}_N).$ 
Actually, on the one side, ~\eqref{estim-trace1} and~\eqref{ineq-tracet} in Lemma~\ref{tech lemma}
are modified as follows: 
$- N |\sigma_x|_\infty^2 |\overline{x}-\overline{y}|\Psi'(|\overline{x}-\overline{y}|)$
(respectively $-\tilde{C}\Psi'(|\overline{x}-\overline{y}|)|\overline{x}-\overline{y}|$)
are replaced by
$-N |\sigma|_{C^{0,\theta}}^2|\overline{x}-\overline{y}|^{2\theta -1}\Psi'(|\overline{x}-\overline{y}|)$
(respectively $-\tilde{C}\Psi'(|\overline{x}-\overline{y}|)|\overline{x}-\overline{y}|^{2\theta -1}$
with $\tilde{C}=\tilde{C}(N,\nu,|\sigma|_\infty, |\sigma|_{C^{0,\theta}})$).
On the other side,
the oscillation bound, Lemma~\ref{oscillation}, hold when
$N|x-y||\sigma_x|_\infty^2$ is replaced by $N|\sigma|_{C^{0,1/2}}^2$ in~\eqref{ssa4}.
The computations in the proofs of Theorems~\ref{uni_grad1} and ~\ref{uni_grad2}
are adapted accordingly.

The proofs of Theorems~\ref{uni_grad1} and ~\ref{uni_grad2} can be adapted easily to quasilinear equations
with diffusion  matrices of type $A(x,p)=\sigma(x,p)\sigma(x,p)^T$ under suitable growth structures in $x,p$ of $\sigma(x,p).$

As far as fully nonlinear equations of Bellman-Isaacs-type
\begin{eqnarray*}
\e v^\e +\mathop{\rm sup}_{a\in A} \mathop{\rm inf}_{b\in B}
\left\{ -{\rm trace}(A_{ab}(x) D^2v^\e)
+  H_{ab}(x, Dv^\e)\right\} =0,
\end{eqnarray*}
are concerned,
our results apply
provided that Assumptions~\eqref{sig-deg},~\eqref{ssa4}-\eqref{LN-ell1},~\eqref{clp}-\eqref{LN-ell2}
hold with constants independent of $a,b.$


\section{Comparison principle, existence and uniqueness for the stationary
equation~\eqref{approx-scal}}\label{NCP}

We prove the well-posedness of the stationary equation in a slightly
more general framework, namely, we work in an open bounded
subset of $\R^N$ instead of $\T^N$, assuming some Dirichlet boundary
conditions hold.

More precisely, we consider
\begin{eqnarray}
\label{Open_Ell}
&& \left\{
\begin{array}{l}
\displaystyle\e v^\e-{\rm trace}(A(x) D^2 v^\e)+H(x,D v^\e)=0, \quad x\in\Omega, 
\\[3mm]
v^\e (x)=g(x),  \qquad x\in \overline{\Omega}.
\end{array}
\right.
\end{eqnarray}
where $\Omega\subset\R^N$ is an open bounded set with $\partial \Omega \in C^{1,1}$,
$g\in C(\partial\Omega),$ $\e >0$ and we need to assume that
$H\in C(\overline{\Omega}\times \R^N;\R)$ to prove the comparison principle. 

The comparison principle follows easily from the ad-hoc inequality~\eqref{estim_global} which follows.
\begin{prop}\label{pri bound}
Assume~\eqref{sig-deg} and
either~\eqref{ssa4}-\eqref{LN-ell1} or~\eqref{clp}-\eqref{LN-ell2} hold,
where the torus $\T^N$ is replaced by $\Omega$. 
Let $u\in USC(\overline{\Omega})$ be a subsolution and 
$v\in LSC(\overline{\Omega})$ be a supersolution of~\eqref{Open_Ell} 
such that $u \le g \le v$ on $\partial\Omega$ and
\begin{eqnarray}\label{estim_bound}
d:=\sup_{x \in \overline{\Omega}}(u-v)>0.
\end{eqnarray}
Then, there exists a constant $C$ such that
\begin{eqnarray}\label{estim_global}
u(x)-v(y)\le d+C|x-y|~~\text{for all $x,y\in \overline{\Omega}$}.
\end{eqnarray}
\end{prop}

The proof of the proposition follows the same ideas of the proof of Theorem~\ref{uni_grad2}. 
We only sketch the minor changes between two proofs.

\begin{proof}
We make the proof under assumptions~\eqref{clp}-\eqref{LN-ell2},
the another case being simpler.
With the assumption $\partial \Omega \in C^{1,1}$ and~~\eqref{clp}, the result of~~\cite{clp10} gives
\begin{eqnarray}\label{clp210}
u \text{ is $\frac{k-2}{k-1}$-H\"older continuous in $\overline{\Omega}$ with a constant $K_0$},
\end{eqnarray}
where $k$ is given by the  assumption~~\eqref{clp}. 

Since $u,v$ are bounded, we can set
\begin{eqnarray}\label{borne}
\mathcal{U}=||u||_{\infty}+||v||_{\infty}.
\end{eqnarray} 

By the upper semi-continuity of $u$ and the compactness of $\partial\Omega,$ 
there exists $r>0$ such that
\begin{eqnarray*}
u(x)-u(y)\le d~~\text{for all $y\in \partial\Omega,x\in \Omega$ and $|x-y|\le r$},\\
v(x)-v(y)\le d~~\text{for all $x\in \partial\Omega,y\in \Omega$ and $|y-x|\le r$}.
\end{eqnarray*}
Hence, using $u\leq v$ on $\partial\Omega,$  there exists $r>0$ such that
\begin{eqnarray*}
u(x)-v(y)\le d~~\text{if $|y-x|\le r$ and either $x\in \partial\Omega$ or $y\in \partial\Omega$}.
\end{eqnarray*}
This implies that for $C=\frac{\mathcal{U}}{r},$ we have
\begin{eqnarray}\label{first estim}
u(x)-v(y)\le d+C|x-y|~~\text{if either $x\in \partial\Omega$ or $y\in \partial\Omega$}.
\end{eqnarray}

{\it Step 1. }Now, we prove that for any $\chi\in (\frac{k-2}{k-1},1)$, there exists a constant $K$ such that
\begin{eqnarray}\label{holder adhoc}
\max_{x,y \in \overline{\Omega}}\{ u(x)-v(y)-d-K|x-y|^{\chi}\} \le 0,
\end{eqnarray} 
where $K>0$ depends only on  $C,\alpha,\beta$ 
given by the hypothesis~\eqref{LN-ell2} and will be precised later.

We argue by contradiction
assuming that the maximum is positive for any $K>0.$ It is therefore achieved at 
$(\overline{x},\overline{y})$ 
with $\overline{x}\not= \overline{y}$. Denote $s:=|\overline{x}-\overline{y}|.$
We have
\begin{eqnarray}\label{ineg183}
Ks^\chi < u(\overline{x})-v(\overline{y})-d.
\end{eqnarray}
It follows from~\eqref{borne} that $s$ tends to zero as $K\to +\infty.$
Thanks to~\eqref{first estim}, we then infer that necessarily 
$\overline{x},\overline{y} \in \Omega$ for $K$ big enough.
Therefore, for $K$ big enough, we can write the viscosity inequalities for
$u$ at $\overline{x}$ and $v$ at $\overline{y}.$

From this point, the next arguments follow exactly the same ones of Part 1 and 2 
in the proof of Theorem~\ref{uni_grad2}. The only minor difference if the way we 
get~\eqref{strange2}. From~\eqref{ineg183} and~\eqref{clp210}, setting $\Psi(t)=Kt^\chi,$
we obtain
\begin{eqnarray*}
s\Psi'(s) \leq Ks^\chi <
u(\overline{x})-u(\overline{y})\le K_0 s^{\frac{k-2}{k-1}},
\end{eqnarray*}
hence 
\begin{eqnarray*}
s(\Psi'(s))^{ k-1} \le K_0^{ k-1}\text{ and }\frac{1}{s}\ge (\frac{K}{K_0})^{ \frac{1}{\chi-\frac{k-2}{k-1}}}.
\end{eqnarray*}
which is exactly the estimation~~\eqref{strange2} as desired. 

\noindent{\it Step 2.  Proof of~\eqref{estim_global}}. 
Consider the function $\Psi(s)=A_1 [A_2s-(A_2s)^{1+\gamma}]$ defined as in~\eqref{fct-conc-lip-hold}
with $r, A_1, A_2>0$ satisfying~\eqref{R1}.
Consider
\begin{eqnarray*}
\max_{x,y \in \overline{\Omega}}\{ u(x)-v(y)-d-\Psi(|x-y|)\}.
\end{eqnarray*}
If the maximum is negative,~\eqref{estim_global} holds with $C=A_1A_2.$
From now, we argue by contradiction
assuming that the maximum is positive and achieved at 
$(\overline{x},\overline{y}).$
With the choice of $r$ in~\eqref{R1}, we have
$0< s:= |\overline{x}-\overline{y}| < r.$
Using the same arguments as in the beginning of Step 1, up to take $A_2$
big enough, we can assume that $\overline{x},\overline{y} \in \Omega$
and therefore we can write the viscosity inequalities for
$u$ at $\overline{x}$ and $v$ at $\overline{y}.$

The next arguments follow exactly the same ones of Part 3 in the proof of Theorem~\ref{uni_grad2}. 
The only minor difference is the way we get~\eqref{strange10}.
Fix any $\chi \in (\frac{k-2}{k-1},1).$ From Step 1, we obtain
\begin{eqnarray*} 
d+\Psi(s) < u(\overline{x})-v(\overline{y}) 
\le d+Ks^\chi.
\end{eqnarray*}
We then have $s\Psi'(s) \le \Psi(s) < K s^\chi$, hence 
\begin{eqnarray*}
s(\Psi'(s))^{ \frac{1}{1-\chi}} \le K^{ \frac{1}{1-\chi}}.
\end{eqnarray*}
This is exactly Estimate~\eqref{strange10} as we want. Having on hands~\eqref{strange10}
we repeat readily the arguments of Part 3 in the proof of Theorem~\ref{uni_grad2}
to conclude.
\end{proof}

We now prove the comparison principle

\begin{thm}\label{CP}
Assume~\eqref{sig-deg}, $H$ is continuous and
either~\eqref{ssa4}-\eqref{LN-ell1} or~\eqref{clp}-\eqref{LN-ell2} hold,
where the torus $\T^N$ is replaced by $\Omega$. 
Let $u\in USC(\overline{\Omega})$ be a subsolution and 
$v\in LSC(\overline{\Omega})$ be a supersolution of~\eqref{Open_Ell} 
such that $u \le g \le v$ on $\partial\Omega.$
Then
\begin{eqnarray*}
u(x)\le v(x) \quad \text{for all $x\in \overline{\Omega}$}.
\end{eqnarray*}
\end{thm}

Notice that we assume that the Dirichlet boundary conditions hold in the classical
viscosity sense on $\partial\Omega$. This is a little bit restrictive especially
when working with superquadratic Hamiltonians since it is known that
loss of boundary conditions may happen, see~\cite{bdl04} for instance. But it is enough
for our purpose here since we work in the periodic setting without boundary
condition.

\begin{proof}
The proof of this result is followed quite easily from the estimate~\eqref{estim_global}. 
Define $d$ as in~\eqref{estim_bound}. We assume that $d>0$ and try to get a contradiction. 
Since $u \le g \le v$ on $\partial\Omega$, any $z\in\overline{\Omega}$ such that
$d=u(z)-v(z)$ lies in $\Omega.$
The maximum
\begin{eqnarray*}
M_\eta=\max_{x,y \in \overline{\Omega}}\{u(x)-v(y)-\frac{|x-y|^2}{2 \eta^2}-d\} \ge 0
\end{eqnarray*}
is achieved
at $(x_\eta,y_\eta)\in \overline{\Omega} \times \overline{\Omega}$.
If there is a sequence $\eta\to 0$ such that $x_\eta,y_\eta \to \overline{x} \in \partial\Omega$,
then 
\begin{eqnarray*}
M_\eta=u(x_\eta)-v(y_\eta)-\frac{|x_\eta-y_\eta|^2}{2 \eta^2}-d \to u(\overline{x})-v(\overline{x})-d<0,
\end{eqnarray*}
which is a contradiction.

Therefore, $(x_\eta,y_\eta) \in \Omega\times \Omega$
for $\eta$ small enough. The theory of second order viscosity solutions yields, for every
$\varrho>0,$ the existence of 
$(p_\eta,X) \in \overline{J}^{2,+}u(x_\eta),(p_\eta,Y) \in \overline{J}^{2,-}v(y_\eta)$
such that
and the following viscosity inequalities hold 
\begin{eqnarray*}
&& \left\{
\begin{array}{ll}
\displaystyle \e u(x_\eta)
-{\rm trace}(A(x_\eta) X)+H(x_\eta,p_\eta)
\le 0,\\[2mm]
\displaystyle \e v(y_\eta)
-{\rm trace}(A(y_\eta)Y)+H(y_\eta,p_\eta) \ge 0.
\end{array}
\right.
\end{eqnarray*}

Thanks to Proposition~\ref{pri bound}, we have 
\begin{eqnarray*}
\frac{|x_\eta-y_\eta|^2}{2\eta^2}+d 
\leq u(x_\eta)-v(y_\eta)
\leq d+C|x_\eta-y_\eta|.
\end{eqnarray*}
Thus
\begin{eqnarray*}
\frac{p_\eta}{2}
= \frac{|x_\eta-y_\eta|}{2\eta^2}
\leq K.
\end{eqnarray*}
This implies that $p_\eta$ is bounded independently of $\eta$. 
Subtracting the viscosity inequalities and using~\eqref{estim-trace1},
we get $\e d \leq H(y_\eta,p_\eta)-H(x_\eta,p_\eta)+O(\eta )+O(\varrho),$
which leads to a contradiction when $\varrho\to 0,$ $\eta\to 0,$
thanks to the uniform continuity of $H$ on compact subsets.
\end{proof}

As a consequence of the previous results, we obtain the well-posedness 
for~\eqref{approx-scal} in the class of Lipschitz continuous functions.

\begin{cor}\label{statio-bien-posee}
Assume~\eqref{sig-deg}, $H\in C(\T^N\times\R^N;\R)$ and
either~\eqref{ssa4}-\eqref{LN-ell1} or~\eqref{clp}-\eqref{LN-ell2} hold.
Then, there exists a unique continuous viscosity solution $v^\e$
of~\eqref{approx-scal} which is Lipschitz continuous
with a constant independent of $\e.$
Moreover, if $A=\sigma\sigma^T$ and $H$ are $C^\infty,$ then
$v^\e$ is $C^\infty.$
\end{cor}

\begin{proof}
Thanks to the comparison principle, Theorem~\ref{CP}, we can construct
a unique continuous viscosity solution to~\eqref{approx-scal}
with Perron's method. To apply this method, it is enough to build
some sub and supersolution to~\eqref{approx-scal} which is easily
done by considering $v^\pm(x)=\pm \frac{1}{\e} |H(\cdot,0)|_\infty.$
The Lipschitz regularity of the solution is then obtained from 
Theorems~\ref{uni_grad1} and~\ref{uni_grad2}.
When $A$ and $H$ are $C^\alpha$ in $x$, the $C^{2,\alpha}$ regularity of $v^\e$
is a consequence of the Lipschitz bounds and the classical
elliptic regularity theory~\cite[Theorems 6.13 and 6.14]{gt98}.
\end{proof}

\section{Applications}\label{Applications}

\subsection{Ergodic problem}
As a first application of Theorems~\ref{uni_grad1} and~\ref{uni_grad2}, we prove
that we can solve the ergodic problem associated with~\eqref{approx-scal},
namely, there exist $v^0\in W^{1,\infty}(\T^N)$ and a unique constant $c\in\R$ 
solution to
\begin{eqnarray}\label{pb-ergod}
-{\rm trace}(A(x) D^2v^0)
+  H(x, Dv^0)=c, & x\in\T^N.
\end{eqnarray}

\begin{thm}\label{thm-erg}
Assume~\eqref{sig-deg}, $H\in C(\T^N\times\R^N;\R)$ and
either~\eqref{ssa4}-\eqref{LN-ell1} or~\eqref{clp}-\eqref{LN-ell2} hold.
Then, there exists $(c,v^0)\in \R\times W^{1,\infty}(\T^N)$ solution to~\eqref{pb-ergod}
and $c$ is unique. If we assume moreover that $H(x,\cdot)$ is locally Lipschitz
continuous then $v^0$ is unique up to additive constants.
\end{thm}

\begin{proof}
Having on hands Theorems~\ref{uni_grad1} and~\ref{uni_grad2}, the result is
an easy application of the method of~\cite{lpv86} and the strong maximum principle.
We only give a sketch of proof.
Let $v^\e$ be the Lipschitz continuous solution of~\eqref{approx-scal}
given by Corollary~\ref{statio-bien-posee}. Since $|\e v^\e|\leq |H(\cdot, 0)|_\infty$
and $|Dv^\e|_\infty\leq K,$ the sequences
$\e v^\e$ and $v^\e-v^\e(0)$ are bounded and equicontinuous
in $C(\T^N)$ for all $\e>0.$ By Ascoli-Arzela Theorem, they converge, up to
subsequences to $-c\in\R$ and $v^0\in W^{1,\infty}(\T^N)$ respectively.
By stability, $(c,v^0)$ is a solution of~\eqref{pb-ergod}.
To prove the uniqueness part of the theorem, assume we have
two solutions $(c_1,v_1)$ and $(c_2,v_2)$ of~\eqref{pb-ergod}.
Then $\tilde{u}_1(x,t):=v_1(x)-c_1t-(|v_1|_\infty+|v_2|_\infty)$
and $\tilde{u}_2(x,t):=v_2(x)-c_2t$ are respectively subsolution
and supersolution of the associated evolution problem~\eqref{edp-evol}
with initial datas $\tilde{u}_1(x,0)\leq  \tilde{u}_2(x,0).$
Since both $\tilde{u}_1$ and $\tilde{u}_2$ are Lipschitz continuous,
we have a straightforward comparison principle for the
evolution problem which yields $\tilde{u}_1(x,t)\leq  \tilde{u}_2(x,t)$
for all $(x,t)\in\T^N\times [0,+\infty).$ Sending $t\to +\infty,$
we infer $c_1\geq c_2$ and exchanging the role of the two solutions,
we conclude $c_1=c_2.$ It is then easy to prove, using the Lipschitz continuity
of $v_1,v_2$ and $H$ with respect to the gradient that $v=v_1-v_2$ is a subsolution of 
$-{\rm trace}(A(x) D^2v)-C|Dv|\leq 0$ in $\T^N$
for some constant $C>0.$ By the strong maximum principle (\cite{dalio04}), $v_1-v_2$ 
is constant.
\end{proof}

\subsection{The parabolic equation}
\label{sec:parab}

In this section, we prove the well-posedness and 
time-independent gradient bounds
for the nonlinear parabolic problem~\eqref{edp-evol}
both under the assumptions~\eqref{ssa4}-\eqref{LN-ell1} 
and~\eqref{clp}-\eqref{LN-ell2}.

\begin{thm}\label{thm-evol-unif}
Assume~\eqref{sig-deg} and that $H\in C(\T^N\times\R^N;\R)$ satisfies
either~\eqref{ssa4}-\eqref{LN-ell1} or~\eqref{clp}-\eqref{LN-ell2}.
For any initial data $u_0\in C^2(\T^N)$, there exists a unique
continuous viscosity solution $u$ to~\eqref{edp-evol} such that,
for all $x,y\in\T^N,$ $s,t\in [0,+\infty),$
\begin{eqnarray}\label{bornes-lip-parab}
|u(x,t)-u(y,s)|\leq K|x-y|+\Lambda|t-s| \quad \text{with $K,\Lambda$ independent of time.}
\end{eqnarray}
If, in addition $A,$ $H$ and $u_0$ are $C^\infty,$ then $u\in C^\infty(\T^N\times [0,+\infty)).$ 
\end{thm}

To prove the theorem, we adapt the proofs of Theorems~\ref{uni_grad1}
and~\ref{uni_grad2}. The proof under the set of 
assumptions~\eqref{clp}-\eqref{LN-ell2} 
is more delicate
since the proof of Theorem~\ref{uni_grad2} requires first
to construct a solution to~\eqref{edp-evol} which is
$\frac{k-2}{k-1}$-H\"older continuous. Due to the lack
of comparison principle for~\eqref{edp-evol} in our case
and since the  H\"older regualrity result of~\cite{clp10}
does not apply directly to evolution equations, the task is difficult.
We need to extend the result of~\cite{clp10} for 
subsolutions of~\eqref{edp-evol} which are Lipschitz continuous in
time (see Lemma~\ref{clp-time}) and to construct an approximate
solution of~\eqref{edp-evol} which is indeed Lipschitz continuous
in time.

\begin{proof}[Proof of Theorem \ref{thm-evol-unif}] \ \\[2mm]
\noindent{\it Step 1. Proof when~\eqref{ssa4}-\eqref{LN-ell1} hold.} 
We truncate the Hamiltonian $H$ by defining
\begin{eqnarray}\label{Hn}
H_n(x,p)= \left\{
\begin{array}{ll}
H(x,p) & x\in\T^N, |p|\leq n,\\
H(x,n\frac{p}{|p|}) &  x\in\T^N, |p|>n.
\end{array}
\right.
\end{eqnarray}
Notice that, on the one side, for $n\geq L,$ $H_n$ satisfies~\eqref{ssa4}.
On the other side, for all $n,$ $H_n$ satisfies~\eqref{LN-ell1} with
the same constant $C$ as for $H.$ Moreover $H_n$ converges locally
uniformly to $H$ as $n\to +\infty.$

By construction, 
$H_n\in BUC(\T^N\times\R^N;\R).$ It follows that the comparison
principle holds for~\eqref{edp-evol} where $H$ is replaced by $H_n.$
Since $H_n(x,Du_0(x))=H(x,Du_0(x)),$ for $n$ large enough,
\begin{eqnarray}\label{sub-sup-sol}
u^\pm (x,t)= u_0(x)\pm |H(\cdot,Du_0)-{\rm trace}(AD^2u_0)|_\infty t
\end{eqnarray}
are respectively super and subsolutions of~\eqref{edp-evol} with $H_n,$
and
Perron's method yields a unique continuous viscosity solution $u_n$
of this latter equation.

By Theorem~\ref{thm-erg}, there exists a
solution $(c_n,v_n)\in \R\times W^{1,\infty}(\T^N)$
of~\eqref{pb-ergod} where $H$ is replaced by $H_n.$ Notice that, since $H_n$
satisfies~\eqref{ssa4}-\eqref{LN-ell1} with constants independent of
$n$ for $n> L,$ both $|v_n|_\infty$ and $|Dv_n|_\infty$ are bounded
independently of $n.$ Choosing $\mathcal{A}$ independent of $n$ such that
$\mathcal{A}\geq |v_n|_\infty + |u_0|_\infty,$
the functions $(x,t)\mapsto v_n(x)-c_n t\pm \mathcal{A}$ are respectively
a viscosity super and subsolutions of~\eqref{edp-evol} with $H_n.$ 
By comparison with $u_n$ we get
\begin{eqnarray*}
v_n(x)-c_n t - \mathcal{A} \leq u_n(x,t)\leq v_n(x)-c_n t + \mathcal{A} 
\quad\text{for all $x\in\T^N,$ $t\in [0,+\infty).$}
\end{eqnarray*}
It follows that 
\begin{eqnarray*}
{\rm osc}(u_n(\cdot,t)) \leq |Dv_n|_\infty {\rm diam}(\T^N) + 2\mathcal{A} \leq C
\end{eqnarray*}
with $C$ independent of $n,t.$

It is now possible to mimic the proof of 
Theorem~\ref{uni_grad1} for $u_n.$

We begin by proving that $u_n$ is 
$\gamma$-H\"older continuous with a constant independent of $t,n$
for some $\gamma\in (0,1).$
For any $\eta >0,$ consider 
\begin{eqnarray*}
M_\eta:= \mathop{\rm max}_{x,y\in\T^N, t>0} \{u_n(x,t)-u_n(y,t)-\Psi(|x-y|)-\eta t\},
\end{eqnarray*}
where $\Psi(s)=Ks^\gamma,$ $0<\gamma <1.$ If the maximum
is nonpositive for some $K>1$ and all $\eta >0,$ then
we are done.
Otherwise, for all $K>1,$ there exists $\eta >0$ such that
the maximum is positive. It is achieved at some 
$(\overline{x},\overline{y},\overline{t})$ with $\overline{x}\not= \overline{y}.$

If $\overline{t}=0,$ then, using that $|\overline{x}- \overline{y}|\leq \sqrt{N},$
we have
\begin{eqnarray*}
M_\eta \leq u_0(\overline{x})-u_0(\overline{y})
-K|\overline{x}- \overline{y}|^\gamma
\leq C_0|\overline{x}- \overline{y}|-K|\overline{x}- \overline{y}|^\gamma
\leq 0
\end{eqnarray*}
for $K>C_0\sqrt{N}^{1-\gamma},$ where $C_0$ is the Lipschitz constant
of $u_0.$ 

It follows that, for $K$ big enough, the maximum is achieved at 
$\overline{t}>0$ and we can write the viscosity inequalities
for $u_n$
using the parabolic version of Ishii's Lemma~\cite[Theorem 8.3]{cil92}. 
Using Lemma~\ref{tech lemma} in
this context, we get
\begin{eqnarray*}
&&\eta -4\nu\Psi''(|\overline{x}-\overline{y}|)
-\tilde{C}\Psi'(|\overline{x}-\overline{y}|)|\overline{x}-\overline{y}|
+H(\overline{x},\Psi'(|\overline{x}-\overline{y}|)q)
-H(\overline{y},\Psi'(|\overline{x}-\overline{y}|)q) < 0.
\end{eqnarray*}
We then obtain a contradiction in the above inequality repeating
readily the proof of Step~1 of Theorem~\ref{uni_grad1} with 
$\mathcal{O}:=\mathop{\rm sup}_{t>0}{\rm osc}(u_n(\cdot,t)).$

With the same adaptations as above in this parabolic context,
we can reproduce the rest of the proof of Theorem~\ref{uni_grad1}.
We conclude that $u_n$ is Lipschitz continuous in space with a
constant independent of $t,n$ since we used~\eqref{ssa4}-\eqref{LN-ell1}
with constants independent of $n$ and since ${\rm osc}(u_n(\cdot,t))$ is
bounded independently of $t,n.$

By Ascoli-Arzela Theorem, up to extract subsequences, $u_n$
converge locally uniformly in $\T^N\times [0,+\infty)$
as $n\to +\infty$
to a function $u$ which is still Lipschitz continuous in space
with a constant independent of $t.$ By stability,
$u$ is a solution to~\eqref{edp-evol}.

The proof of the Lipschitz continuity of $u$ in time requires
$u_0$ to be $C^2$ and can be done exactly as 
in the second case below.
\medskip

\noindent{\it Step 2. Proof when~\eqref{clp}-\eqref{LN-ell2} hold.} 
We consider, for $q,n\geq 1,$ the approximate problem
\begin{eqnarray}\label{edp-app}
&&\left\{
\begin{array}{ll}
u_t-{\rm trace}(A(x) D^2u)
+ \frac{1}{q}|Du|^M+ H_n(x, Du)=0, & (x,t)\in\T^N\times (0,+\infty),\\
u(x,0)=u_0(x), &  x\in\T^N,
\end{array}
\right.
\end{eqnarray}
where $M>2$ 
and $H_n$ is defined in~\eqref{Hn}.

We have a comparison principle for~\eqref{edp-app} since $H_n\in BUC(\T^N\times\R^N)$
and $\frac{1}{q}|p|^M$ is a nonlinearity which is independent of $x$;
when subtracting the viscosity inequality, this term disappears since
we are in $\T^N$ and there is no need to add a localization term in the test-function to
achieve the maximum. Moreover, since~\eqref{sub-sup-sol} are
still super and subsolutions of~\eqref{edp-app}, by means of Perron's
method, we can build a continuous viscosity solution $u^{qn}$ of
the problem~\eqref{edp-app}.

The next lemma extends the result of~\cite{clp10} for USC subsolutions
of parabolic equations with coercive Hamiltonian satisfying~\eqref{clp}.
The proof is postponed at the end of the section.
\begin{lem}\label{clp-time}
Assume that~\eqref{clp} holds.
Let $U\in USC(\T^N\times [0,+\infty))$ be a subsolution of~\eqref{edp-evol}
which is bounded and Lipschitz continuous in time with constants independent of $t.$
Then, there exists $\tilde{C}>0$ which depends on $k,A,\Lambda$ (appearing in~\eqref{clp}
and~\eqref{lipent}) but not on $t$
such that
\begin{eqnarray*}
|U(x,t)-U(y,t)|\leq \tilde{C}|x-y|^{\frac{k-2}{k-1}}
\quad x\in\T^N, t\geq 0.
\end{eqnarray*}
\end{lem}

We are going to prove that $u^{qn}$ satisfies the assumptions of Lemma~\ref{clp-time}.
We first claim that there exists a constant $c^{qn}$ bounded with respect to $n$
such that $u^{qn}+c^{qn}t$ is bounded in $\T^N\times [0,+\infty)$ by a constant
depending on $q$ but not on $n.$
The equation
\begin{eqnarray}\label{app-statio}
\e v -{\rm trace}(A(x) D^2 v)
+ \frac{1}{q}|Dv|^M+ H_n(x, Dv)=0, 
\quad x\in\T^N,
\end{eqnarray}
satisfies Assumptions~\eqref{clp}-\eqref{LN-ell2}
of Theorem~\ref{uni_grad2} with $k=M$ and a constant $C$
depending on $q$ but not on $n.$ By Theorem~\ref{thm-erg},
there exists a
solution $(c^{qn},v^{qn})\in \R\times W^{1,\infty}(\T^N)$
of the associated ergodic problem.
By the maximum principle, $|\e v|\leq |H_n(\cdot , 0)|_\infty \leq 
|H(\cdot , 0)|_\infty$ so $c^{qn}$ is bounded independently of $q,n.$
Moreover, since the constants in the assumptions in Theorem~\ref{uni_grad2} may be
taken independent of $n,$ $v^{qn}$ is bounded and Lipschitz continuous
with constants independent on $n.$ 
Noticing that 
$\tilde{v}^{qn}(x,t)= v^{qn}(x)-c^{qn} t\pm A_q$ are respectively
viscosity super and subsolutions of~\eqref{edp-app}
when $A_q\geq |v^{qn}|_\infty + |u_0|_\infty$ ($A_q$ may be chosen independent
of $n$). By comparison with $u^{qn}$ we get
\begin{eqnarray*}
v^{qn}(x)-c^{qn} t - A_q \leq u^{qn}(x,t)\leq v^{qn}(x)-c^{qn} t + A_q 
\quad\text{for all $x\in\T^N,$ $t\in [0,+\infty)$}
\end{eqnarray*}
and the claim is proved.

We then claim that  $u^{qn}$ is Lipschitz continuous in time, i.e., there
exists $\Lambda>0$ independent of $t,q,n$ such that
\begin{eqnarray}\label{lipent}
|u^{qn}(x,t)-u^{qn}(x,s)|\leq \Lambda |x-y| \quad x\in\T^N, s,t\geq 0.
\end{eqnarray}
The proof is classical and relies on the comparison principle together
with the fact that $u_0\in C^2(\T^N).$ We only give a sketch of proof.
Since $A$ and the Hamiltonian in~\eqref{edp-app} do not depend on $t,$
for all $h>0,$ $u^{qn}(\cdot,\cdot+h)$ is solution to~\eqref{edp-app}
with initial data $u^{qn}(\cdot,h).$ By comparison, we obtain
\begin{eqnarray}\label{ineg391}
u^{qn}(x,t+h)-u^{qn}(x,t)\leq \mathop{\rm sup}_{y\in\T^N} (u^{qn}(y,h)-u_0(y))^+
\quad x\in\T^N, t\geq 0.
\end{eqnarray}
Setting
\begin{eqnarray*}
\Lambda:= |{\rm trace}(A D^2u_0)|_\infty + |Du_0|_\infty^M + |H(\cdot, 0)|_\infty
\end{eqnarray*}
(notice that $\Lambda$ does not depend neither on $q$ nor $n$),
we have that $u_0(x)\pm \Lambda t$ are respectively super and subsolutions
of~\eqref{edp-app}. By comparison, it follows
$|u^{qn}(x,t)-u_0(x)|\leq \Lambda t.$ Using this inequality in~\eqref{ineg391},
we obtain~\eqref{lipent}.

Therefore, we can apply Lemma~\ref{clp-time} to $U(x,t)= u^{qn}(x,t)+c^{qn}t$
which is Lipschitz continuous in time with a constant independent of $t,q,n$
since $c^{qn}$ is bounded independently on $q,n.$
We obtain that $u^{qn}(x,t)+c^{qn}t$ and so
$u^{qn}$ is $\frac{M-2}{M-1}$-H\"older
continuous in space with a constant depending on $q$ (but not on $n,t$).
By Ascoli-Arzela Theorem, $u^{qn}$ converges, up to subsequences, 
locally uniformly
in $\T^N\times [0,+\infty)$ as $n\to +\infty$ to
a function $u^{q}$ which still satisfies~\eqref{lipent}
(with the same constant $\Lambda$). Moreover, by stability,
$u^{q}$ is solution to~\eqref{edp-app} with $H_n$
replaced by $H.$

Arguing as above on~\eqref{app-statio} where $H_n$
is replaced by $H,$ we can construct a solution
$(c^q,v^q)$ to the ergodic problem associated
to~\eqref{app-statio} with $H.$ Using that
\begin{eqnarray}\label{coerc-fine}
\frac{1}{q}|p|^M + H(x,p)\geq \frac{1}{C}|p|^k -C
\end{eqnarray}
this time and that~\eqref{LN-ell2} holds
for datas independent of $q,$ we can prove
that $c^q$ is bounded and $v^q$ is bounded and Lipschitz continuous
with constants independent of $q.$
By comparison, $u^{q}+c^q t$ is bounded independently of $q,t.$

Applying again Lemma~\ref{clp-time} to $u^{q}+c^q t$
but using~\eqref{coerc-fine}, we obtain that $u^{q}$ 
is  $\frac{k-2}{k-1}$-H\"older
continuous with a constant independent of $q$ now.
Thanks again to Ascoli-Arzela Theorem, we can send $q\to +\infty$
to obtain, up to subsequences, a solution $u$ of~\eqref{edp-evol}
which is still  $\frac{k-2}{k-1}$-H\"older
continuous with a constant independent of $t.$

We are not in position to mimic the proof of Theorem~\ref{uni_grad2}
for this solution $u$, which is done easily adapting the proof in the
time-dependent case.

In conclusion, we built a Lipschitz continuous (in space and time)
solution to~\eqref{edp-evol} with constants independent of $t.$
\medskip

\noindent{\it Step 3. Uniqueness in  the class of continuous functions
and upper regularity.}
Even if a strong comparison principle between semicontinuous
viscosity sub and supersolutions does not necessarily hold
for~\eqref{edp-evol} under our assumptions,
it is easy to see that a comparison principle holds 
if either the subsolution
or the supersolution is Lipschitz continuous. It allows
to compare any continuous viscosity solution of~\eqref{edp-evol}
with $u.$ 

The regularity of $u$ when the data $u_0 \in C^{2,\alpha}$ and $H$ is $C^{\alpha}$ in $x$-variable is a consequence
of the Lipschitz bounds and the classical parabolic regularity theory, see~~\cite{krylov96} for instance.

The proof of the theorem is complete. 
\end{proof}

\begin{rem}
When $\sigma \in C^{0,1/2}(\T^N;\mathcal{M}_N)$
instead of $W^{1,\infty}(\T^N;\mathcal{M}_N),$
we need to regularize also $\sigma$ into
a Lipschitz continuous matrix to build a continuous
solution. The estimates on the approximate solutions
are not affected by this regularization thanks to
the results of Sections~\ref{os gra} and~\ref{NCP} and the result
of~\cite{clp10},  which
hold for $\sigma \in C^{0,1/2}(\T^N;\mathcal{M}_N).$
\end{rem}

\begin{proof}[Proof of Lemma~\ref{clp-time}]
To prove the lemma, it is sufficient to prove that
there exists $C>0$ such that, for
every $t>0,$ 
\begin{eqnarray}\label{ineg-en-x}
&&-{\rm trace}(A(x) D^2U(x,t))
+  H(x, DU(x,t))\leq C
\quad \text{for $x\in\T^N$ in the viscosity sense.}
\end{eqnarray}
Indeed, once~\eqref{ineg-en-x} is established, we can repeat
readily the proof of~\cite[Theorem 2.7]{clp10}.

Fix $t>0$ and suppose that $x_0\in\T^N$ is a strict
maximum point of $x\mapsto U(x,t)-\varphi(x)$ in $\T^N,$
where $\varphi\in C^2(\T^N).$ The supremum
\begin{eqnarray*}
\mathop{\rm sup}_{x,y\in\T^N, t\geq 0} \{ U(x,s)-\varphi(x)-\frac{(t-s)^2}{\eta^2}\}
\end{eqnarray*}
is achieved at $(\overline{x},\overline{s})$ and, since $U$ is bounded,
$\frac{(t-\overline{s})^2}{\eta^2}\to 0$ and $\overline{x}\to x_0$
as $\eta\to 0.$ Writing that $(\overline{x},\overline{s})$ is a maximum
point we have
\begin{eqnarray*}
U(\overline{x},t)-\varphi(\overline{x})
\leq U(\overline{x},\overline{s})-\varphi(\overline{x})- \frac{(t-\overline{s})^2}{\eta^2}.
\end{eqnarray*}
Using the Lipschitz continuity with respect to time of $U$ (let us say with constant 
$\Lambda$ independent of $t$), we obtain
\begin{eqnarray}\label{ineq836}
\frac{|t-\overline{s}|}{\eta^2}\leq \Lambda.
\end{eqnarray}
Since $U$ is a viscosity subsolution of~\eqref{edp-evol}, we get
\begin{eqnarray*}
\frac{\overline{s}-t}{\eta^2}
-{\rm trace}(A(\overline{x}) D^2\varphi(\overline{x}))
+  H(\overline{x}, D\varphi(\overline{x}))\leq 0.
\end{eqnarray*}
Taking into account~\eqref{ineq836} and letting $\eta\to 0,$
we infer
\begin{eqnarray*}
-{\rm trace}(A(x_0) D^2\varphi(x_0))
+  H(x_0, D\varphi(x_0))\leq \Lambda,
\end{eqnarray*}
which proves~\eqref{ineg-en-x}.
\end{proof}

We end this section with a general bound for the oscillation of
continuous solutions to~\eqref{edp-evol} when
the comparison result holds. It is the analogous
of Lemma~\ref{oscillation} in the parabolic setting and is a result
interesting by itself. We give below as an easy application
the convergence of $u(x,t)/t$ towards a constant.

\begin{lem}\label{para oscillation}
Suppose that comparison principle holds for~\eqref{edp-evol}.
Let $u_0\in C^2(\T^N)$ and assume that $H, A, u_0$ satisfy
\begin{eqnarray}\label{ssa-parab}
&&  \left\{\begin{array}{l}
\text{there exists $L>1$ such that
for all $x,y\in\T^N,$
if $|p|\!= \!L$, then }\\
\displaystyle  H(x,p) \ge|p|\left[H(y,\frac{p}{|p|})\!+\!
|H(\cdot ,Du_0)-{\rm trace}(A D^2u_0)|_\infty
\!+\! N|x-y||\sigma_x|_\infty^2\right].
\end{array}\right.
\end{eqnarray}
Then, the unique continuous solution $u$ of~\eqref{edp-evol} satisfies
\begin{eqnarray}\label{os para}
&& u(x,t)-u(y_t,t) \le L|x-y_t|, \quad\text{ for all $t\geq 0,$ $x \in \T^N$}
\end{eqnarray} 
and $y_t$ such that $\displaystyle u(y_t,t)=\min_{x \in \T^N}u(x,t)$.
\end{lem}

Notice that~\eqref{ssa-parab} is a parabolic version of~\eqref{ssa4} 
which holds as soon as $H$ is superlinear.

\begin{rem}
Assuming that the comparison principle holds is a bit restrictive in this
context but we do not succeed to skip it.
\end{rem}

\begin{proof}[Proof of Lemma~\ref{para oscillation}]
Setting
\begin{eqnarray}\label{expli const}
\mathcal{A}:= |H(\cdot ,Du_0)-{\rm trace}(A D^2u_0)|_\infty,
\end{eqnarray}
we have that $u_0(x)\pm \mathcal{A}t$ are respectively super and subsolutions
of~\eqref{edp-evol}. By comparison, it follows
$|u(x,t)-u_0(x)|\leq \mathcal{A}t.$ 
By comparison again, we get
\begin{eqnarray}\label{lip in time}
|u(x,t+s)-u(x,t)|\leq \mathcal{A}s.
\end{eqnarray}

Fix $T>0$. We define 
\begin{eqnarray*}
M=\max_{x,y \in \T^N, t \in [0,T]}\{u(x,t)-L u(y,t)+(L-1)\min_{x \in \T^N}u(x,t)-L|x-y|\},
\end{eqnarray*}
where the constant $L$ is the one in~\eqref{ssa-parab}. 
If $M\leq 0$, then~\eqref{os para} is straightforward. 
Otherwise, $M \ge L\delta>0$ for $\delta >0$ enough small. 

Thanks to~\eqref{lip in time}, we can approximate $\phi (t):=\min_{x \in \T^N}u(x,t)$ from 
below over the compact interval $[0,T]$ by a sequence of smooth functions $\phi_n(t)$ 
whose lipschitz norm is bounded by $\mathcal{A}$ given by~\eqref{expli const}. 
Up to choosing $n$ big enough, we may assume $0 \le \phi-\phi_n \le \delta$. 
For $n \in \N$, we consider
\begin{eqnarray*}
M_n=\max_{x,y \in \T^N\!,\, t \in [0,T]}\{ u(x,t)-L u(y,t)+(L-1)\phi_n(t) -L|x-y|\}.
\end{eqnarray*}
It is clear that $M_n \ge \delta>0$.
The above positive maximum is achieved at 
$(x_n,y_n,t_n)$ with $x_n\not= y_n$. Unless $u(x_n,t_n)-L u(x_n,t_n)+(L-1)\phi_n(t_n)\ge \delta$, 
which is impossible since $\phi_n(t)\le \phi(t)=\min_{x \in \T^N}u(x,t).$ Moreover, by replacing $L$ with $\max\{L,||Du_0||_\infty\}$ if necessary, we can see easily that $t_n>0$.
The claim is proved and the maximum in $M_n$ is achieved at a differentiable
point of the test-function.

The theory of second order viscosity solutions \cite[Theorem 8.3]{cil92}
yields,  for every $\varrho>0,$ the existence of 
$(a,p,X) \in \overline{J}^{2,+}u(x_n,t_n)$ and
$(b/L,p/L,Y/L) \in \overline{J}^{2,-}u(y_n,t_n)$, with 
$p=L \frac{x_n-y_n}{|x_n-y_n|}$, $a-b=-(L-1)\phi'(t_n)$, 
such that
\begin{eqnarray*}
&& \left\{
\begin{array}{ll}
\displaystyle a
-{\rm trace}(A(x_n) X)+H(x_n,p)
\le 0,\\[3mm]
\displaystyle  \frac{b}{L}-{\rm trace}(A(y_n)\frac{Y}{L})+H(y_n,\frac{p}{L}) \ge 0.
\end{array}
\right.
\end{eqnarray*}
It follows
\begin{eqnarray*}
&& -(L-1)\phi_n'(t_n)  
-{\rm trace}(A(x_n) X -A(y_n)Y) +H(x_n,p)-L H(y_n,\frac{p}{L})
\leq 0.\nonumber
\end{eqnarray*}
Using Lemma~\ref{tech lemma}, we have
\begin{eqnarray*}
&&-{\rm trace}(A(x_n)X
-A(y_n)Y)\geq 
- LN|x_n-y_n| |\sigma_x|_\infty^2+O(\varrho)
\end{eqnarray*}
Finally, we obtain
\begin{eqnarray*}
H(x_n,p)
-L \left[H(y_n,\frac{p}{L}) +N|x_n-y_n| |\sigma_x|_\infty^2\right]+O(\varrho)
\leq (L-1)\phi_n'(t_n) < L\mathcal{A}.
\end{eqnarray*}
Letting $\varrho\to 0$ and applying~\eqref{ssa-parab} yields a contradiction.
\end{proof}

We end this section by an application of the oscillation bound.

\begin{prop}\label{u/t}
Assume~\eqref{ssa-parab} and
suppose that a comparison principle for~\eqref{edp-evol}
holds. For every $u_0\in C(\T^N),$
there exists $c\in\R$ such that the unique solution $u$
of~\eqref{edp-evol} satisfies
\begin{eqnarray*}
\lim_{t \to \infty}\frac{u(x,t)}{t}= -c~~~\text{uniformly with respect to $x \in \T^N$}.
\end{eqnarray*}
\end{prop}

For related results in the case of Bellman equations, see~\cite{al98, ab10}.

\begin{proof}[Sketch of proof of Proposition~\ref{u/t}]
Without loss of generality, we assume that $u_0\in C^2(\T^N)$. The general case where  $u_0\in C(\T^N)$ can be handled using an approximation of $u_0$ in the class of $C^2$ functions and the comparison principle.

Set $m(t)=\mathop{\rm min}_{\T^N} u(\cdot,t).$
Since $(x,t)\mapsto u_0(x)-\mathcal{A}t,$ where $\mathcal{A}$
is given by~\eqref{expli const}, is a subsolution of~\eqref{edp-evol}, we have
$m(t)\geq -C(1+t).$ Moreover, an easy application of the
comparison principle yields that $m$ is subadditive, namely
$m(t+s)\leq m(t)+m(s)$ for all $t,s\geq 0.$
By the subadditive theorem, there exists $c\in\R$
such that $m(t)/t\to -c$ as $t\to +\infty.$
By Lemma~\ref{para oscillation}, $0\leq u(x,t)-m(t)\leq L\,{\rm diam}(\T^N).$
This implies the uniform convergence of $u(\cdot,t)/t$ to $-c.$
\end{proof}

\subsection{Large time behavior of solutions of nonlinear strictly
parabolic equations}\label{ep ltb}
In this section, we use the uniform gradient bound proved in Theorems~\ref{uni_grad1} and~\ref{uni_grad2} 
to study the large time behavior of the solution of~\eqref{edp-evol}.

The first results on the large time behavior of solutions for second order parabolic equaions were established in
Barles-Souganidis~\cite{bs01}. They prove the uniform gradient bounds~\eqref{grad-intro}
and~\eqref{gradt-intro} for~\eqref{approx-scal} and~\eqref{edp-evol} in two cases. 
The first one is
for Hamiltonians with a sublinear growth 
with respect to the gradient. A typical example is
\begin{eqnarray}\label{sublin}
H(x,p)=\langle b(x),p\rangle+\ell(x), \quad \text{$b\in C(\T^N;\R^N),$ $\ell \in C(\T^N).$}
\end{eqnarray}
The second case is
for {\em superlinear Hamiltonians}. The precise assumptions (\cite[(H2)]{bs01}) are more involved
and require both local Lipschitz regularity properties and convexity-type assumptions
on $H.$ These assumptions are designed to allow the use of weak Bernstein-type
arguments (\cite{barles91a}). The typical example is
with a superlinear growth with respect to the gradient
\begin{eqnarray}\label{typic-super-intro}
H(x,p)=a(x)|p|^{1+\alpha}+\ell(x), \quad \text{$\alpha >0,$ $a,\ell \in W^{1,\infty}(\T^N)$ and $a>0.$}
\end{eqnarray}
The proof of the large time behavior of the solution of~\eqref{edp-evol}
is then a consequence of the strong maximum principle
(we give a sketch of proof below).

On the one hand, our resuts generalizes the assumptions on sublinear Hamiltonians made in ~\cite{bs01}. More importantly, our results allow to deal with a class of superlinear Hamiltonians which is very different
with the superlinear case of~\cite{bs01}. 

\begin{thm}(Large time behavior)\label{LTB}
Assume that either the assumptions of Theorem~\ref{uni_grad1} or the assumptions of Theorem~\ref{uni_grad2} hold. 
Moreover, suppose that $H$ is continuous and locally lipschitz with respect to $p$. 
Then, there exists a unique $c\in\R$ such that, for all $u_0\in C(\T^N),$ 
the solution $u$ of~\eqref{edp-evol} satisfies
\begin{eqnarray}\label{asympt-form} 
u(x,t)+ct\to v^0(x) \quad uniformly \ as \ t\to +\infty,
\end{eqnarray}
where $(c,v^0)$ is a solution of~\eqref{pb-ergod}.
\end{thm}

\begin{proof}[Sketch of proof of Theorem \ref{LTB}] 
First of all, it is enough to assume that $u_0\in C^2(\T^N)$. The general case where  $u_0\in C(\T^N)$ can be handled using an approximation of $u_0$ in the class of $C^2$ functions and the comparison principle.

Set $m(t)=\max_{x\in\T^N}(u(x,t)+ct-v^0(x)).$ By the comparison principle, $m$ is 
nonincreasing and, since it is bounded from below, $m(t)\to \ell \text{ as } t \to \infty.$
From Theorem~\ref{thm-evol-unif}, 
$\{u(\cdot,t)+ct, t>0\}$ is relatively compact in $W^{1,\infty}(\T^N).$
So we can extract a sequence,
$t_j\to +\infty$ such that $u(\cdot,t_j)+ct_j\to  \bar{u}\in W^{1,\infty}(\T^N).$ 
Applying the comparison principle for~\eqref{edp-evol} 
in $W^{1,\infty}(\T^N\times [0,+\infty)),$ we obtain,
for every  $x\in\T^N,$ $t\geq 0,$ $p\in\N,$
\begin{eqnarray*}
&& |u(x,t+t_j)+ct_j-u(x,t+t_p)-ct_p)|\leq 
\max_{y\in\T^N} |u(y,t_j)+ct_j-u(y,t_p)-ct_p|,
\end{eqnarray*}
which proves that $(u(\cdot ,\cdot+t_j)+c(\cdot+t_j))_j$ is a Cauchy
sequence in $C(\T^N\times [0,+\infty)).$ We call $u_\infty$
its limit. Notice, on  one hand, that  
$|Du_{\infty}(\cdot,t)|_\infty\leq K$ for all $t$
and, on the other hand, that $u_\infty -ct$ is solution of~\eqref{edp-evol}
with initial data $\bar{u}$ by stability.

Passing to the limit with respect to $j$ in $m(t+t_j)$ 
we obtain 
\begin{eqnarray}\label{maxxt}
\ell=\max_{x}(u_{\infty}(x,t)-v^0(x))\quad \text{for any $t>0$.}
\end{eqnarray}

Since  $u_{\infty}$ is solution of~\eqref{edp-evol} ith $c$ in the
right-hand side
and $v^0$ is solution of~\eqref{pb-ergod}, 
thanks to the Lipschitz continuity of $u_\infty, v^0$ with respect to $x$ and 
$H$ with respect to the gradient, we obtain that there exists $C>0$
such that
$w=u_\infty -v^0$ is subsolution of
$w_t-{\rm trace}(A(x) D^2w)-C|Dw|\leq 0$ in $\T^N\times [0,+\infty).$
Using~\eqref{maxxt} and
the strong maximum principle (\cite{dalio04}), we infer 
$u_{\infty}(x,t)-v^0(x)=\ell$ for every
$(x,t) \in \T^N \times [0,+\infty).$
Noticing that $\ell+v^0(x)$ does not depend on the choice 
of subsequences, we obtain
$u(x,t) +ct-\ell-v^0(x)\to 0$ uniformly in $x$ as $t \to \infty$.
\end{proof}

\subsection{Existence result of H\"older continuous solutions for 
equations without comparison principle}\label{vanishing idea}
Usually, existence results for Equations like~\eqref{approx-scal} or~\eqref{edp-evol}
are consequence of a strong comparison principle as Theorem~\ref{CP} together with Perron's method
or using the value function of an optimal control problem when $H$ is
convex. In this section, we use Theorem~\ref{CP} and the result of~\cite{clp10} to 
build H\"older continuous solutions under assumptions which are too weak to
expect any comparison principle.

\begin{thm}\label{existence}
Assume~$A\ge 0$, $H$ is continuous and satisfies
\begin{eqnarray}
\label{poly H}
\frac{|p|^m}{C}-C \le H(x,p)\le C(|p|^M+1), \quad x\in\T^N, p \in \R^N, 2<m\leq M.
\end{eqnarray}
Then there exists a viscosity solution $v^\e$ of~\eqref{approx-scal}
which is $\frac{m-2}{m-1}$-H\"older continuous solution 
and, for every $u_0\in C^2(\T^N),$ 
a viscosity solution $u$ of~\eqref{edp-evol} which is
$\frac{m-2}{m-1}$-H\"older continuous in space and Lipschitz continuous in $t.$
\end{thm}

\begin{proof}
The proof follows the approach used in Step 2 of the proof of Theorem~\ref{thm-evol-unif}.
\smallskip

\noindent{\it Step 1. Existence for the stationary problem~\eqref{approx-scal}.}
Equation~\eqref{approx-scal} with $H$ replaced by 
$H_q(x,p)=\frac{|p|^{M+1}}{q}+H(x,p)$ and $A$ replaced by $A+\frac{1}{q} I$
satisfies the conditions of Theorem~\ref{CP}, 
hence we have the strong comparison principle for this new equation. 
Therefore, we can apply Perron's method to obtain the existence of a 
continuous solution $v_q^\e$.
From~\cite{clp10}, $v_q^\e$ is $\frac{m-2}{m-1}$-H\"older 
continuous. Using Ascoli-Arzela Theorem and stability when $q\to +\infty,$ we obtain
the existence of a viscosity solution $v^\e$
which is  $\frac{m-2}{m-1}$-H\"older continuous (with a constant independent
of $\e$).
\smallskip

\noindent{\it Step 2. Existence of H\"older continuous solutions to the
ergodic problem.} 
We can reproduce the beginning of the proof of Theorem~\ref{thm-erg}
with $v^\e$: the sequences $\e v^\e$ and $v^\e-v^\e(0)$
are still equicontinuous and therefore, we can build a solution
$(c,v^0)\in \R\times C^{0, \frac{m-2}{m-1}}(\T^N)$ to~\eqref{pb-ergod}.
\smallskip

\noindent{\it Step 3. Existence for the parabolic problem.}
We now consider~\eqref{edp-app}. 
This equation satisfies a strong comparison
principle. We can follow readily the proof of Step 2 of Theorem~\ref{thm-evol-unif}
up to obtain a H\"older continuous solution $u^q.$
Notice it is possible to
build a solution to~\eqref{pb-ergod} as explained in Step 2 above. 
The comparison of $u^q$ with $v^q-c^q t \pm C$ where $C$ is a big constant
is not anymore straightforward as in the proof of Theorem~\ref{thm-evol-unif}
since $v^q$ is only H\"older continuous and not Lipschitz continuous.
To continue, we need to adapt the proof of Theorem~\ref{CP} to the parabolic
case which can be done easily since $u^q, v^q$ are $\frac{m-2}{m-1}$-H\"older continuous
in space. It is then possible to send a subsequence $q\to +\infty$ to obtain a H\"older
continuous (in space) solution $u$ to~\eqref{edp-evol} as desired.
\end{proof}


\section{Appendix}

\begin{proof}[Proof of Lemma \ref{oscillation}]
For simplicity, we skip the $\e$ superscript in $v^\e$. 
The constant $L$ which appears below is the one of~\eqref{ssa4}. 
Consider
\begin{eqnarray*}
M=\max_{x,y \in \T^N}\{ v(x)-L v(y)+(L-1)\min v -L|x-y|\}.
\end{eqnarray*}
We are done if $M\le 0$.
Otherwise, the above positive maximum  is achieved at 
$(\overline{x},\overline{y})$ with $\overline{x}\not= \overline{y}$.
Notice that the continuity of $v$ is crucial at this step. 
The theory of second order viscosity solutions 
yields, for every $\varrho>0,$ the existence of 
$(p,X) \in \overline{J}^{2,+}v(\overline{x})$ and
$(p/L,Y/L) \in \overline{J}^{2,-}v(\overline{y})$, $p=L \frac{\bar{x}-\bar{y}}{|\bar{x}-\bar{y}|}$,
such that
\begin{eqnarray*}
&& \left\{
\begin{array}{ll}
\displaystyle \e v(\bar{x})
-{\rm trace}(A(\bar{x}) X)+H(\bar{x},p)
\le 0,\\[2mm]
\displaystyle  \e v(\overline{y})-{\rm trace}(A(\overline{y})\frac{Y}{L})+H(\overline{y},\frac{p}{L}) \ge 0.
\end{array}
\right.
\end{eqnarray*}
Using Lemma~\ref{tech lemma}, we have
\begin{eqnarray*}
&&-{\rm trace}(A(\overline{x})X
-A(\overline{y})Y)\geq 
- LN|x-y| |\sigma_x|_\infty^2+O(\varrho)
\end{eqnarray*}

It follows
\begin{eqnarray*}
  \hspace*{0.2cm}  \e (v(\bar{x})-Lv(\overline{y}))
-{\rm trace}(A(\overline{x}) X
-A(\overline{y})Y)
+H(\bar{x},p)
-L H(\overline{y},\frac{p}{L})\}
\leq 0.\nonumber
\end{eqnarray*}

Recall that  $\e\min v\leq |H(\cdot ,0)|_\infty$, then
\begin{eqnarray*}
\e (v(\bar{x})- Lv(\overline{y}))>-(L-1)\e\min v
\ge -L|H(\cdot ,0)|_\infty.
\end{eqnarray*}

Finally, we obtain
\begin{eqnarray*}
H(\bar{x},p)
-L \left[H(\overline{y},\frac{p}{L}) +|H(\cdot ,0)|_\infty+N|x-y| |\sigma_x|_\infty^2\right]<0.
\end{eqnarray*}
Applying \eqref{ssa4} yields a contradiction.
\end{proof}

\subsection{Proof of Lemma~\ref{tech lemma}}
For simplicity, we skip the $\e$ superscript in $v^\e$.
The theory of second order viscosity solutions yields (see \cite[Theorem 3.2]{cil92} for instance), 
for every $\varrho>0,$ the existence of 
$(p,X) \in \overline{J}^{2,+}v(\overline{x}),(p,Y) \in \overline{J}^{2,-}v(\overline{y})$
such that~\eqref{mat},~\eqref{mat-bis},~\eqref{mat-ter} hold. 

Le us prove~\eqref{estim-trace1} and~\eqref{ineq-tracet}.
From~\eqref{mat}, for every $\zeta, \xi\in\R^N,$ we have
\begin{eqnarray*}
\langle X\zeta,\zeta \rangle - \langle Y\xi,\xi\rangle
\leq \Psi'\langle \zeta-\xi,B(\zeta-\xi) \rangle 
+ \Psi'' \langle \zeta-\xi,(q\otimes q)(\zeta -\xi)\rangle
+O(\varrho).
\end{eqnarray*}
We estimate ${\rm trace}(A(\overline{x})X)$
and ${\rm trace}(A(\overline{y})Y)$ using
two orthonormal bases $(e_1,\cdots ,e_N)$ and
$(\tilde{e}_1,\cdots ,\tilde{e}_N)$ in the following way:
\begin{eqnarray}\label{estimT}
T:= {\rm trace}(A(\overline{x})X
-A(\overline{y})Y)
&=&
\sum_{i=1}^N \langle X\sigma(\overline{x}) e_i, \sigma(\overline{x}) e_i \rangle
- \langle Y\sigma(\overline{y}) \tilde{e}_i, 
\sigma(\overline{y}) \tilde{e}_i \rangle
\nonumber\\
&\leq &
\sum_{i=1}^N \Psi'\langle \zeta_i,B\zeta_i\rangle
+ \Psi''\langle \zeta_i,(q\otimes q)\zeta_i\rangle+O(\varrho)
\nonumber\\
&\leq &
\Psi''\langle \zeta_1,(q\otimes q)\zeta_1\rangle
+ \sum_{i=1}^N \Psi'\langle \zeta_i,B\zeta_i\rangle+O(\varrho),
\end{eqnarray}
where we set $\zeta_i= \sigma(\overline{x}) e_i-\sigma(\overline{y}) \tilde{e}_i$
and noticing that 
$\Psi''\langle \zeta_i,(q\otimes q)\zeta_i\rangle =\Psi'' \langle \zeta_i, q\rangle^2
\leq 0$ since $\Psi$ is concave.

We now build a suitable base to prove~\eqref{estim-trace1}
and another one to prove~\eqref{ineq-tracet}.

In the case of~\eqref{estim-trace1} where $\sigma$ could be degenerate, 
we choose any orthonormal basis such that $e_i= \tilde{e}_i.$ It follows
\begin{eqnarray*}
T&\leq & 
\sum_{i=1}^N \Psi'\langle (\sigma(\overline{x})-\sigma(\overline{y}))e_i 
,B(\sigma(\overline{x})-\sigma(\overline{y}))e_i\rangle+O(\varrho)\\
&\leq &
\Psi' N |\sigma(\overline{x})-\sigma(\overline{y})|^2|B|+O(\varrho)\\
&\leq & \Psi' N |\sigma_x|_\infty^2 |\overline{x}-\overline{y}|+O(\varrho)
\end{eqnarray*}
since $|B|\leq 1/|\overline{x}-\overline{y}|.$ Thus~\eqref{estim-trace1}
holds.

When~\eqref{sig-deg} holds, i.e.,
$A(x)\geq \nu I$ for every $x,$ the matrix $\sigma(x)$ is invertible and
we can set
\begin{eqnarray*}
e_1= \frac{\sigma(\overline{x})^{-1}q}{|\sigma(\overline{x})^{-1}q|},
\quad \tilde{e}_1= -\frac{\sigma(\overline{y})^{-1}q}{|\sigma(\overline{y})^{-1}q|},
\quad \text{where $q$ is given by~\eqref{mat-bis}}.
\end{eqnarray*}
If $e_1$ and $\tilde{e}_1$ are collinear, 
then we complete the basis with orthogonal unit vectors 
$e_i=\tilde{e}_i\in e_1^\perp,$ $2\leq i\leq N.$ Otherwise, in the plane
${\rm span}\{e_1, \tilde{e}_1\},$ we consider a rotation $\mathcal{R}$
of angle $\frac{\pi}{2}$ and define
\begin{eqnarray*}
e_2=\mathcal{R}e_1, \quad 
\tilde{e}_2 =-\mathcal{R}\tilde{e}_1.
\end{eqnarray*}
Finally, noticing that ${\rm span}\{e_1, e_2\}^\perp
={\rm span}\{\tilde{e}_1, \tilde{e}_2\}^\perp,$ we can complete the orthonormal
basis with unit vectors $e_i=\tilde{e}_i\in {\rm span}\{e_1, e_2\}^\perp,$ $3\leq i\leq N.$

From~\eqref{sig-deg}, we have
\begin{eqnarray}\label{ineg-sig}
\nu \leq \frac{1}{|\sigma(x)^{-1}q|^2}\leq |\sigma|_\infty^2.
\end{eqnarray}
 It follows
\begin{eqnarray*}
\langle \zeta_1,(q\otimes q)\zeta_1\rangle =
\left( \frac{1}{|\sigma(\overline{x})^{-1}q|}
+\frac{1}{|\sigma(\overline{y})^{-1}q|} \right)^2\geq 4\nu.
\end{eqnarray*}
From~\eqref{mat-bis}, we deduce $Bq=0.$ 
Therefore
\begin{eqnarray*}
\langle \zeta_1,B\zeta_1\rangle =0.
\end{eqnarray*}
For $3\leq i\leq N,$ we have
\begin{eqnarray*}
\langle \zeta_i,B\zeta_i\rangle
= \langle (\sigma(\overline{x})-\sigma(\overline{y}))e_i,
B(\sigma(\overline{x})-\sigma(\overline{y}))e_i\rangle
\leq |\sigma_x|_\infty^2 |\overline{x}-\overline{y}|.
\end{eqnarray*}

Now, we estimate $\zeta_2$
\begin{eqnarray*}
|\zeta_2|
= |(\sigma(\overline{x})-\sigma(\overline{y}))\mathcal{R}e_1 +
\sigma(\overline{y})\mathcal{R}(e_1+\tilde{e}_1)|
\leq |\sigma_x|_\infty|\overline{x}-\overline{y}| + |\sigma|_\infty |e_1+\tilde{e}_1|.
\end{eqnarray*}
It remains to estimate 
\begin{eqnarray*}
|e_1+\tilde{e}_1|
&\leq& \frac{1}{|\sigma(\overline{x})^{-1}q|}
|\sigma(\overline{x})^{-1}q-\sigma(\overline{y})^{-1}q|
+ |\sigma(\overline{y})^{-1}q|
\left|\frac{1}{|\sigma(\overline{x})^{-1}q|}-\frac{1}{|\sigma(\overline{y})^{-1}q|}\right|\\
&\leq &
\frac{2|\sigma|_\infty |\sigma_x|_\infty}{\nu}  |\overline{x}-\overline{y}|,
\end{eqnarray*}
from~\eqref{ineg-sig} and $|(\sigma^{-1})_x|_\infty\leq  |\sigma_x|_\infty/\nu.$

From~\eqref{estimT}, we finally obtain
$T\leq 4\nu\Psi'' +\tilde{C}\Psi'|\overline{x}-\overline{y}|+O(\varrho)$
where
\begin{eqnarray}\label{def-ctilde}
&& \tilde{C}=\tilde{C}(N,\nu,|\sigma|_\infty, |\sigma_x|_\infty)
:= |\sigma_x|_\infty^2(N-2 +(1+\frac{2|\sigma|_\infty^2}{\nu})^2).
\end{eqnarray}
This completes the proof of~\eqref{ineq-tracet}.

We finally prove~\eqref{estimation-outil}.
Writing the viscosity inequality
for the subsolution $v$ of~\eqref{approx-scal} at $\overline{x}$ and 
the supersolution $v$ at $\overline{y},$
we get
\begin{eqnarray*}
&& \left\{
\begin{array}{ll}
\displaystyle \e v(\bar{x})
-{\rm trace}(A(\bar{x}) X)+H(\bar{x},p)
\le 0,\\[2mm]
\displaystyle \e v(\overline{y})
-{\rm trace}(A(\overline{y})Y)+H(\overline{y},p) \ge 0.
\end{array}
\right.
\end{eqnarray*}
Since the maximum is supposed to be positive and $\Psi\geq 0,$ we have
$v(\bar{x})> v(\overline{y})$ and obtain
\begin{eqnarray}
\label{visco-ineq185}
-{\rm trace}(A(\overline{x}) X
-A(\overline{y})Y)+H(\bar{x},p)
-H(\overline{y},p)
< 0.\nonumber
\end{eqnarray}
Estimate~\eqref{estimation-outil} follows from a straightforward application
of~\eqref{ineq-tracet}.



\end{document}